\documentclass[a4paper,11pt]{amsart}

\usepackage{amssymb} 
\usepackage{xcolor} 
\usepackage{fancybox}
\usepackage{dashbox}
\newboolean{showcomments}
\setboolean{showcomments}{true}
\ifthenelse{\boolean{showcomments}}
{ \newcommand{\mynote}[3]{
  \fbox{\bfseries\sffamily\scriptsize#1}
  {\small$\blacktriangleright$\textsf{\emph{\color{#3}{#2}}}$\blacktriangleleft$}}}
{ \newcommand{\mynote}[3]{}}
\definecolor{asparagus}{rgb}{0.53, 0.66, 0.42}

  \usepackage{float}

  \usepackage{pgffor}
  \usepackage{tikz}
  \usetikzlibrary{calc,3d,arrows,positioning,shapes.misc}

  \usepackage[english]{babel}
  \usepackage{amsthm}
  \usepackage{amsmath}
  \usepackage{amssymb}
  \usepackage[latin1]{inputenc}
  \usepackage{setspace}
  \usepackage{enumerate}
  \usepackage{stmaryrd}

  \usepackage[colorlinks=true,linkcolor=black,citecolor=black,filecolor=black,urlcolor=black,menucolor=black]{hyperref}

  \usepackage{array}
  \usepackage{graphicx}
  \usepackage[babel]{csquotes}
  \usepackage{geometry}
  \usepackage{bbm}
  \usepackage{stmaryrd}
  \usepackage[all]{xy}
  \usepackage{mathrsfs}

  \theoremstyle{plain}
  \newtheorem{theorem}{Theorem}
  \newtheorem{proposition}{Proposition}

  \newtheorem{lemma}[proposition]{Lemma}
  \theoremstyle{definition}
  \newtheorem{definition}{Definition}

	\allowdisplaybreaks[3]

  \def \E{\mathscr{E}}
  
  \def \H{\mathcal{H}}

  \def \R{\mathbb{R}}

  \newcommand {\supp} {\mathop \textup{supp}}

 \newcommand {\dist} {\mathop \textup{dist}}
 
\DeclareMathOperator*{\esssup}{ess\,sup}
  \newcommand {\p} {\partial}
  \newcommand {\eps} {\varepsilon}
 \newcommand {\jump}[1] {[\![ #1 ]\!]}

  \title[{BV} solutions for {MCF} with constant contact angle]
	{{BV} solutions for mean curvature flow with constant contact angle: Allen--Cahn 
	approximation and weak-strong uniqueness}
	
\author{Sebastian Hensel}
\address{Institute of Science and Technology Austria (IST Austria), Am~Campus~1, 
3400 Klosterneuburg, Austria}
\email{sebastian.hensel@ist.ac.at}
\curraddr{Hausdorff Center for Mathematics, Universit{\"a}t Bonn, Endenicher Allee 62, 53115 Bonn, Germany
(\texttt{sebastian.hensel@hcm.uni-bonn.de})}

\author{Tim Laux}
\address{Hausdorff Center for Mathematics, Universit{\"a}t Bonn, Endenicher Allee 62, 53115 Bonn, Germany}
\email{tim.laux@hcm.uni-bonn.de}
    \date{\today}

\begin{document}

\begin{abstract}
    	We study weak solutions to mean curvature flow satisfying Young's angle condition
			for general contact angles~$\alpha \in (0,\pi)$. 
    	First, we construct BV solutions using the Allen--Cahn approximation with boundary contact energy as proposed by Owen and Sternberg.
    	Second, we prove the weak-strong uniqueness and stability for this solution concept.
    	The main ingredient for both results is a relative energy, which can also be interpreted as a tilt excess. 
    	
  \medskip 	
    

  \medskip

  \end{abstract}
\maketitle

\tableofcontents

\section{Introduction}

	\subsection{Context}
	The evolution of embedded surfaces by mean curvature flow (MCF) arises in many physical systems in which surface tension effects are dominant.
	A typical boundary condition for such surfaces is a prescribed angle at which the surface meets the boundary of a given container.
	This angle, called the Young angle, is dictated by the surface tensions, as the surface meets the boundary at a fixed angle which is energetically optimal.
	In this work, we propose and study a weak solution concept
	to MCF satisfying such a boundary condition, following ideas 
	introduced by Luckhaus and Sturzenhecker \cite{Luckhaus1995} for the whole space setting. As these solutions are based on functions of bounded variation, we refer to them as BV solutions. 
	We are then interested in two problems: the construction of BV solutions for MCF with
	constant contact angle, in particular using the canonical Allen--Cahn approximation; and uniqueness properties of such BV solutions, in particular as long as a classical solution to mean curvature flow with constant contact angle exists.
	
	The free energy to model boundary contact in the phase-field framework of the Allen--Cahn equation was  proposed by Cahn~\cite{Cahn1977}; see \eqref{eq:def Eeps} below for the precise formula.
	The behavior of minimizers in the sharp-interface limit was investigated by Modica \cite{Modica1987} in the framework of $\Gamma$-convergence.
	Building on their previous work \cite{Kagaya2017},  Kagaya and Tonegawa~\cite{Kagaya2018} analyzed critical points in the sharp-interface limit using the theory of varifolds.

	The dynamical model we study here was introduced by Owen and Sternberg \cite{Owen1992}, who also carried out formal matched asymptotic expansions. 
	It is simply the $L^2$-gradient flow of Cahn's energy on a slow time scale, see \eqref{eq:AllenCahn1}--\eqref{eq:AllenCahn3} below.
	Based on the comparison principle, Katsoulakis, Kossioris and Reitich~\cite{Katsoulakis1995} proved the convergence to the viscosity solution in a convex container.
	In the case of zero boundary energy, Abels and Moser~\cite{Abels2019} derived convergence rates before the onset of singularities. 
	Their proof relies on a spectral gap inequality and a Gronwall argument to make an asymptotic expansion as in \cite{Owen1992} rigorous. 
	The same authors extended this result to contact angles close to $90^\circ$ \cite{Abels2021}. 
	As their method does not rely on the comparison principle, Moser~\cite{Moser2021} was able to generalize the proof to the vectorial Allen--Cahn equation in the case of a transition between two wells.
	

	Our first main result concerns the construction of BV solutions to mean curvature flow using the Allen--Cahn equation.
	Neglecting the effects of boundary contact, Simon and one of the authors~\cite{LauxSimon18} derived a conditional convergence result for the Allen--Cahn equation to a BV solution. 
	Such conditional convergence results are inspired by Luckhaus and Sturzenhecker~\cite{Luckhaus1995} and have reemerged in the multiphase setting by Otto and one of the authors~\cite{Laux-Otto}. 
	Here, we generalize the result~\cite{LauxSimon18} to incorporate boundary effects. 
	When neglecting boundary conditions, the energy convergence can be dropped, see~\cite{Hensel2021e}. We would expect that a similar strategy might also work in the context of the present paper.
	
	Our second main result establishes the weak-strong uniqueness of BV solutions with boundary contact.
	Here, we rely on the notion of calibrated flows, which is a generalization of calibrations to the dynamic setting introduced in our work \cite{Fischer2020a} with Fischer and Simon.
	More precisely, we use a generalization of this concept to the case of boundary contact and show that any calibrated flow is unique in the class of BV solutions.
	The construction of these calibrations in $d=2$ is carried out by Moser and the first author~\cite{Hensel2021c}, who use them to adapt the convergence proof of Fischer, Simon and the second author~\cite{Fischer2020}. 
	We expect that combining ideas from~\cite{Hensel2021c} and our recent work on double bubbles in three dimensions~\cite{Hensel2021} allows to construct these calibrations for smooth mean curvature flows in the presence of boundary contact also in dimension $d=3$.

%
%



	\subsection{Precise setting and assumptions.}
	Let $\Omega\subset \R^d $ be a bounded domain with smooth 
	and orientable boundary $\p \Omega$, and let $T \in (0,\infty)$ be a finite time horizon.
	Following Owen and Sternberg~\cite{Owen1992}, we consider the following Allen--Cahn problem
	\begin{align}
	\label{eq:AllenCahn1}
		\p_t u_\eps &= \Delta u_\eps - \frac1{\eps^2} W'(u_\eps) && \text{in } \Omega \times (0,T),
		\\ \label{eq:AllenCahn2}
		u_\eps &= u_{\eps,0} &&\text{on } \Omega\times\{t=0\},
		\\ \label{eq:AllenCahn3}
		(\nu_{\p \Omega} \cdot  \nabla)u_\eps &= \frac{1}{\eps}\sigma'(u_\eps)&&\text{on } \p \Omega \times (0,T).
	\end{align}
	Here, $\nu_{\p \Omega}$ denotes the inward-pointing unit normal along~$\p\Omega$.
	The nonlinearities~$W \colon \R \to [0,\infty)$ and~$\sigma \colon \R \to [0,\infty)$
	are two (at least) differentiable functions, and we assume $W(s)=0$ if and only if $s \in \{\pm 1\}$.
	More precise assumptions on~$W$ and~$\sigma$ will be given later.
	Here, $W$ is the potential energy in the bulk and $\sigma$ is the energy per area on the boundary of the container. 
	For simplicity, we will focus on the standard double-well potential $W(s)=\frac12(1-s^2)^2$.
	We will state the precise assumptions on the function $\sigma$ later. 
	The standard example to keep in mind is
	\begin{align*}
		\sigma(s)=
		\begin{cases} 
			0, & s\leq -1,\\
			(s-\frac13 s^3+\frac23)\cos \alpha, & s \in [-1,1],\\
			\frac43\cos \alpha,&  s\geq 1,
		\end{cases}
	\end{align*}
	where $\alpha \in (0,\frac\pi2]$ is the desired Young's angle in the sharp-interface limit. 
	
	Our proofs are based on the gradient-flow structure of the system \eqref{eq:AllenCahn1}--\eqref{eq:AllenCahn3}.
	The energy in the gradient-flow structure, first proposed by Cahn~\cite{Cahn1977}, contains a bulk term and a boundary term
	\begin{align}\label{eq:def Eeps}
		E_\eps(u_\eps) = \int_\Omega  \Big( \frac{\eps}{2} |\nabla u_\eps|^2 + \frac1\eps W(u_\eps)  \Big)  dx
		+  \int_{\p\Omega} \sigma(u_\eps) \,d\H^{d-1}.
	\end{align}
	The metric in the gradient-flow structure is simply the standard $L^2$-metric in the bulk (with a small prefactor):
	\begin{align}
		(\delta u_\eps , \delta u_\eps)_{\eps} = \int_\Omega \eps (\delta u_\eps)^2 \,dx.
	\end{align}
	The small prefactor in the metric simply corresponds to a change of variables in time meaning that we are considering the gradient flow on a slow time scale.
	This structure can be read off the energy-dissipation relation
	\begin{align}\label{eq:dtE}
	\frac{d}{dt} E_\eps (u_\eps) = -\int_\Omega \eps (\p_t u_\eps)^2\,dx.
	\end{align}
	This means that the boundary condition \eqref{eq:AllenCahn2} results in an additional term in the energy but does not change the metric.

	For later use, we record the $\Gamma$-limit $E$ of $E_\eps$ due to Modica~\cite{Modica1987}, which states that in the sharp-interface limit, three surface-energy terms compete against each other:
		\begin{align}
			E(u) =  c_0 \H^{d-1}(\p^* \{u{=}1\} \cap \Omega) 
			+ \hat\sigma_+ \H^{d-1}(\p^*\{u{=}1\} \cap \p\Omega)
			+ \hat\sigma_- \H^{d-1} (\p^*\{u{=}-1\} \cap \p\Omega)
		\end{align}
		for $u \colon \Omega \to \{\pm 1\} \in BV(\Omega)$, and $E(u) = +\infty$ otherwise.
		The associated surface tension constants $c_0$, $\hat\sigma_+$ and $\hat\sigma_-$ 
		are given as follows:
		\begin{align}
		\label{def:psi}
			\psi(s) &:= \int_{-1}^s \sqrt{2W(s')} \,ds', \quad s\in\R,
			\\
			\label{def:surfaceTensionInterface}
			 c_0 &:= \psi(1) - \psi(-1) = \psi(1) > 0,
		\end{align}
		and
		\begin{align*}
			\hat\sigma_\pm := \hat\sigma(\pm 1) \geq 0,
		\end{align*}
		where 
		\begin{align}
		\label{def:relaxedBoundaryDensity}
			\hat\sigma(s) := \inf \{  \sigma(s') + |\psi(s')-\psi(s)|\colon s'\in \R  \},
			\quad s\in\R.
		\end{align}
		The structure of $\hat \sigma$ is as follows. The function $\hat\tau:=\hat\sigma\circ \psi^{-1}$ is the (lower) 
		$1$-Lipschitz envelope of $\tau:=\sigma \circ \psi^{-1}$:
		\begin{align}
		\label{eq:LipschitzEnvelope}
			\hat \tau = \sup\{\tilde \tau : \tilde \tau \text{ is $1$-Lipschitz and } \tilde \tau \leq \tau\}.
		\end{align}
			
Since $\Omega$ is bounded, by replacing $u_\eps$ by $-u_\eps$, we may w.l.o.g.\ assume $\hat\sigma_+ \geq \hat\sigma_-$ 
and by subtracting the irrelevant constant $\int_{\p\Omega} \hat\sigma_-\,d\H^{d-1}$ 
from $E_\eps$ and $E$, we may w.l.o.g.\ assume $\hat\sigma_-=0$. 
In particular,
\begin{align}
\label{eq:jumpSurfaceTensionContainer}
\jump{\hat \sigma} 
:= \hat\sigma_+ - \hat\sigma_- = \hat\sigma_+ \geq 0.
\end{align}
To allow both the phases $\{u{=}1\}$ and $\{u{=}-1\}$ of the limit problem
to wet the boundary of the container, we assume 
\begin{align}
\hat\sigma_+ < c_0.
\end{align} 
In particular, there exists an angle $\alpha \in (0,\frac{\pi}{2}]$ such that
Young's law holds true, i.e.,
\begin{align}
\label{eq:Young}
c_0\cos\alpha = \hat\sigma_+ = \jump{\hat\sigma}.
\end{align}

Let us state next
the precise assumptions on~$W$ and~$\sigma$ which will be assumed throughout the rest of the paper. 
Since our main attention does not lie on the well-studied bulk energy, we restrict to the standard double-well potential~$W \colon \R \to [0,\infty)$ given by
\begin{align}
\label{assump:W}
\tag{A1}
W(s) = \frac12 (1-s^2)^2, \quad s \in \R.
\end{align}
Regarding the boundary energy density~$\sigma \colon \R \to [0,\infty)$, we consider the following class. In terms of regularity, $\sigma$ is required to satisfy 
\begin{align}
\label{assump:Sigma1}
\tag{A2}
\sigma \in C^{1,1}(\R), \quad \sigma' > 0 \text{ in } (-1,1),
\quad \supp\sigma' \subset [-1,1]. 
\end{align}
Moreover, there exist an angle $\alpha \in (0,\frac{\pi}{2}]$
and a constant $\kappa \in (0,1{-}\cos\alpha)$ such that
(recall the definition~\eqref{def:psi} of~$\psi$)
\begin{align}
\label{assump:Sigma2}
\tag{A3}
(\cos\alpha)\psi \leq \sigma \leq (1{-}\kappa)\psi
\quad\text{on } [-1,1].
\end{align}
Last but not least, we require compatibility at the phase values~$\pm 1$
in form of
\begin{align}
\label{assump:Sigma3}
\tag{A4}
\sigma(-1) = 0, \quad \sigma(1) = c_0 \cos\alpha.
\end{align}
Let us remark that it holds
\begin{align}
\label{eq:relaxedBoundaryDensity}
\hat \sigma = \sigma \quad\text{under the assumptions
\eqref{assump:Sigma1}--\eqref{assump:Sigma3}},
\end{align}
so that we indeed recover~\eqref{eq:jumpSurfaceTensionContainer}
and~\eqref{eq:Young} under these assumptions. The asserted identity $\hat \sigma = \sigma$
in turn follows from straightforward arguments which we leave to the interested reader. 

The rest of the paper is organized as follows. In Section~\ref{sec:results}, we state the main results of this paper. The proof of the first main result, the convergence of the Allen--Cahn equation with nonlinear Robin boundary condition to mean curvature flow with fixed contact angle, is given in Section~\ref{sec:conv}. In Section~\ref{sec:uniqueness}, we prove the second main result, which concerns the uniqueness of mean curvature flow with fixed contact angle. 

\section{Main results} \label{sec:results}

\subsection{Existence of $BV$ solutions to MCF with constant contact angle}
We start by providing the definition of a suitable weak solution concept
which is phrased in the language of sets of finite perimeter.

	\begin{definition}
	\label{def:weak sol}
		Let $d\geq 2$, consider a finite time horizon $T \in (0,\infty)$,
		and let the angle $\alpha\in (0,\frac{\pi}{2}]$ be given by Young's law
		\begin{align}
		\label{eq:YoungsLaw}
		c_0\cos\alpha = \jump{\hat\sigma}.
		\end{align}
		We say a one-parameter family of open sets $A(t) \subset \Omega$
		with finite perimeter in $\R^d$, $t\in[0,T]$, 
		is a \emph{distributional (or $BV$) solution to mean curvature flow in $\Omega$ 
		with constant contact angle $\alpha$} if 
		\begin{enumerate}[1.]
			\item (Existence of normal velocity)
				There exists a $(\H^{d-1} \llcorner (\p^*A(t) \cap \Omega))dt$-measurable function $V$ such that 
				\begin{align}
					\int_0^T\int_{\p^* A(t) \cap \Omega} V^2 \,d\H^{d-1} dt < \infty
				\end{align}
				and $V(\cdot,t)$ is the normal speed of $\p^* A(t)\cap \Omega$ 
				with respect to $-\nu_{A(t)} := -\smash{\frac{\nabla\chi_{A(t)}}{|\nabla\chi_{A(t)}|}}$ in the sense that 
				for almost every $T'\in (0,T)$ and all $\zeta \in C_c^\infty (\overline{\Omega} \times [0,T))$
				\begin{align}
				\label{eq:evolSet}
				\int_{A(T')} \zeta(\cdot,T') \,dx - \int_{A(0)} \zeta(\cdot,0) \,dx
				=	\int_0^{T'} \int_{ A(t)} \p_t \zeta\,dx dt 
				+ \int_0^{T'} \int_{\p^*A(t)\cap \Omega} V \zeta \,d\H^{d-1} dt.
				\end{align}
				
			\item (Motion law) 
			For almost every $T'\in (0,T)$ and any test vector field 
			$B \in C_c^\infty(\overline{\Omega} \times [0,T);\R^d)$ with $B \cdot \nu_{\p \Omega} =0$ 
			on $\p\Omega \times (0,T)$ it holds
			\begin{equation}
			\begin{aligned}
			\label{eq:motionLaw}
				& c_0 \int_0^{T'} \int_{\p^*A(t)\cap \Omega} 
				(I_d {-} \nu_{A(t)}\otimes \nu_{A(t)}) : \nabla B \,d\H^{d-1} dt 
				\\&
				+\jump{\hat \sigma} \int_0^{T'} \int_{\p^*A(t)\cap \p\Omega} 
				(I_d {-} \nu_{\p\Omega}\otimes \nu_{\p \Omega}) : \nabla B \,d\H^{d-1} dt 
				=  \int_0^{T'} \int_{\p^*A(t) \cap \Omega} B\cdot \nu_{A(t)} V \,d\H^{d-1} dt.
			\end{aligned}
			\end{equation}
			
			\item (Optimal energy dissipation rate) We have
			\begin{align}
			\label{eq:energyDissipation}
			 	E(A(T')) +  c_0\int_{0}^{T'} \int_{\p^*A(t)\cap \Omega} 
				V^2 \,d\H^{d-1} dt \leq E(A(0))
			\end{align}		
			for almost every $T' \in (0,T)$, where we  defined
			\begin{align}
			\label{eq:energySharpInterface}
			E(A(t)) := E(2\chi_{A(t)}{-}1) = c_0 \int_{\p^* A(t) \cap \Omega} 1 \,d\H^{d-1} 
			+ \jump{\hat \sigma} \int_{\p^* A(t) \cap \p\Omega} 1\,d\H^{d-1}.
			\end{align}
		\end{enumerate}
	\end{definition}
	
The first main result of the present contribution is concerned with the existence
of distributional solutions to MCF with constant contact angle in the sense of the
previous definition. More precisely, we show that solutions of the Allen--Cahn 
problem~\eqref{eq:AllenCahn1}--\eqref{eq:AllenCahn3} converge subsequentially and
conditionally towards such distributional solutions.
	
	\begin{theorem}
	\label{theo:existence}
	  \emph{i)}
		Let $T \in (0,\infty)$ be a finite time horizon, $d \geq 2$,
		and let $(W,\sigma)$ at least be subject to the assumptions~\eqref{assump:W} 
		and~\eqref{assump:Sigma1}. Moreover, let a sequence $(u_{\eps,0})_{\eps > 0}$
		of initial phase fields be given such that
		\begin{align}
		\label{eq:uniformBoundInitialEnergy}
		\sup_{\eps > 0} E_\eps(u_{\eps,0}) &< \infty,
		\\
		\label{eq:uniformLinftyBoundInitialData}
		\sup_{\eps > 0} \|u_{\eps,0}\|_{L^\infty(\Omega)} &\leq 1,
		\end{align}
		and such that there exists a set of finite perimeter $A(0) \subset \Omega$
		satisfying
		\begin{align}
		\label{eq:convergenceInitialCondition}
		\psi(u_{\eps,0}) &\to c_0\chi_{A(0)} 
		&&\text{in } L^1(\Omega) \text{ as } \eps\downarrow 0,
		\\ \label{eq:initialEnergyConvergence}
		E_\eps(u_{\eps,0}) &\to E(2\chi_{A(0)}{-}1) = E(A(0)) &&\text{as } \eps\downarrow 0,
		\end{align}
		where $\psi$ and $c_0$ are defined in~\eqref{def:psi} 
		and~\eqref{def:surfaceTensionInterface}, respectively.
		Let $(u_\eps)_{\eps > 0}$ denote the associated sequence of weak solutions
		to~\emph{\eqref{eq:AllenCahn1}--\eqref{eq:AllenCahn3}} in the sense
		of Definition~\ref{def:weakSolAllenCahn} below (cf.\ Lemma~\ref{lem:propertiesWeakSol}
		below for existence and uniqueness).
		
		Then there exists a subsequence $\eps \downarrow 0$ and a one-parameter family
		of sets of finite perimeter $A(t) \subset \Omega$, $t \in [0,T]$, such that 
		\begin{align}
		\label{eq:convergenceTheorem}
		\psi_\eps &\to \psi_0 := \psi(2\chi_A{-}1) = c_0\chi_A
		&&\text{strongly in } L^1(\Omega\times(0,T)) \text{ as } \eps \downarrow 0,
		\end{align}
		where $\chi_A(x,t) := \chi_{A(t)}(x)$ for all $(x,t) \in \Omega \times [0,T]$.
		The evolving indicator function satisfies
		$\chi_A\in C([0,T];L^1(\Omega)) \cap BV(\Omega \times (0,T))$.
		
		\emph{ii)}
		Suppose in addition that the assumptions~\eqref{assump:Sigma2} and~\eqref{assump:Sigma3}
		on the boundary energy density~$\sigma$ are satisfied, and that we have
		(with respect to the above subsequence $\eps \downarrow 0$ and the map $\chi_A$)
		\begin{align}
		\label{assump:energyConvergence}
			\lim_{\eps \downarrow0} \int_0^T E_\eps(u_\eps(\,\cdot\,,t))\,dt  = \int_0^TE(A(t))\,dt.
		\end{align} 
		Then $A(t)$, $t \in [0,T]$, is a distributional solution to mean curvature flow in $\Omega$ 
		with contact angle~$\alpha$ according to Definition~\ref{def:weak sol} above.
	\end{theorem}

\subsection{Weak-strong uniqueness of $BV$ solutions to MCF with constant contact angle}	
The second main contribution of the present work is concerned with
the question of uniqueness and stability of $BV$ solutions to MCF with constant contact angle.
To this end, we start by recalling the key ingredient to our approach, namely
the notion of a calibrated evolution. This concept was first introduced
in the work~\cite{Fischer2020a} in the multiphase setting. In the present work,
we are only concerned with the setting of two phases. However, since we 
allow for contact of the interface with the boundary of the container
at a fixed angle, additional boundary conditions have to be encoded.
This has already been done in the recent work by Moser and the first author,
cf.\ \cite[Definition~2]{Hensel2021c}.
For convenience of the reader, we restate the definition.  

\begin{definition}
\label{def:calibratedEvolution}
Let $T\in (0,\infty)$ be a time horizon, and consider a one-parameter family
of open subsets $\mathscr{A}(t) \subset \Omega$, $t \in [0,T]$. Assume that
for each~$t\in [0,T]$, the set~$\mathscr{A}(t)$ is of finite perimeter in~$\R^d$,
and that there exists a finite family of $\H^{d{-}2}$-rectifiable sets 
$\Gamma_c(t) \subset \p\Omega$, $c\in\mathcal{C}$, such that the closure
of~$\p^*\mathscr{A}(t)$ is given by~$\p\mathscr{A}(t)$ and we have the
disjoint decomposition $\p\mathscr{A}(t) = (\p^*\mathscr{A}(t) {\cap} \Omega)
\cup (\p^*\mathscr{A}(t) {\cap} \p\Omega) \cup \bigcup_{c\in\mathcal{C}} \Gamma_c(t)$.
Denoting for~$t\in[0,T]$ by~$\chi_{\mathscr{A}(t)}$ the indicator function of~$\mathscr{A}(t)$,
we further assume that $\chi(x,t):=\chi_{\mathscr{A}(t)}(x)$
satisfies $\chi\in C([0,T];L^1(\R^d)) \cap BV(\R^d \times (0,T))$. Finally,
we assume that $\partial_t\chi \ll |\nabla\chi| \llcorner \big(\bigcup_{t\in [0,T]}
(\p^*\mathscr{A}(t) {\cap} \Omega) {\times} \{t\} \big)$.

The one-parameter family $\mathscr{A}(t)$, $t\in [0,T]$, is
called a \emph{calibrated evolution} with respect to the $L^2$-gradient
flow of the interfacial energy~\eqref{eq:energySharpInterface}
if there exists a tuple of maps~$(\xi,B,\vartheta)$ (which then
is referred to as an associated \emph{boundary adapted gradient flow calibration}) 
satisfying the following: First, the tuple~$(\xi,B,\vartheta)$ is regular in the sense of
\begin{align}
\xi &\in C^1\big(\overline{\Omega}{\times}[0,T];\R^d\big)
\cap C\big([0,T];C^2_{b}(\Omega;\R^d)\big),
\\ \label{eq_regularityB}
B &\in C\big([0,T];C^1(\overline{\Omega};\R^d)\big)
\cap C\big([0,T];C^2_{b}(\Omega;\R^d)\big),
\\ \label{eq_regularityWeight}
\vartheta &\in C^1_{b}\big(\Omega{\times}[0,T];[-1,1]\big) 
\cap C\big(\overline{\Omega}{\times}[0,T];[-1,1]\big).
\end{align}
Second, the pair of vector fields~$(\xi,B)$ satisfies the boundary conditions
\begin{align}
\label{eq:angleCondXi}
\xi(\cdot,t) \cdot \nu_{\p\Omega} &= \cos\alpha, 
&& t \in [0,T],
\\
\label{eq:angleCondB}
B(\cdot,t) \cdot \nu_{\p\Omega} &= 0,
&& t \in [0,T],
\end{align}
with the angle~$\alpha\in(0,\frac{\pi}{2}]$ being given by Young's law~\eqref{eq:YoungsLaw},
whereas the weight~$\vartheta$ is subject to the sign conditions
\begin{align}
\label{eq:signWeightExterior}
\vartheta(\cdot,t) &> 0 
&& \text{in the essential exterior of } \mathscr{A}(t) 
\text{ within } \Omega, \, t \in [0,T],
\\ \label{eq:signWeightInterior}
\vartheta(\cdot,t) &< 0 
&& \text{in the essential interior of } \mathscr{A}(t), \, t \in [0,T],
\\ \label{eq:signWeightInterface}
\vartheta(\cdot,t) &= 0 
&& \text{on } \p^*\mathscr{A}(t) \cap \Omega, \, t \in [0,T].
\end{align}
Third, the vector field~$\xi$ satisfies the coercivity conditions
\begin{align}
\xi(\cdot,t) &= \nu_{\p^*\mathscr{A}(t)}
&& \text{on } \p^*\mathscr{A}(t) \cap \Omega,
\\ \label{eq:coercivityLengthXi}
|\xi|(\cdot,t) &\leq 1 - c\min\{1,\mathrm{dist}^2(\cdot,\overline{\p^*\mathscr{A}(t) \cap \Omega})\}
&& \text{in } \Omega,
\end{align}
whereas the weight~$\vartheta$ is subject to the coercivity condition
\begin{align}
\min\{\dist(\cdot,\p\Omega),\dist(\cdot,\overline{\p^*\mathscr{A}(t) \cap \Omega}),1\}
&\leq C|\vartheta|(\cdot,t)
&& \text{in } \Omega
\end{align}
for some constants~$c\in (0,1)$ and~$C>0$ and all~$t \in [0,T]$. Fourth, 
the tuple~$(\xi,B,\vartheta)$ is subject to the approximate evolution equations
\begin{align}
\label{eq:timeEvolXi}
|\partial_t\xi + (B\cdot\nabla)\xi + (\nabla B)^\mathsf{T}\xi|(\cdot,t)
&\leq C\min\{1,\dist(\cdot,\overline{\p^*\mathscr{A}(t) \cap \Omega})\}
&& \text{in } \Omega,
\\ \label{eq:timeEvolLengthXi}
|\xi \cdot (\partial_t\xi + (B\cdot\nabla)\xi)|(\cdot,t)
&\leq C\min\{1,\mathrm{dist}^2(\cdot,\overline{\p^*\mathscr{A}(t) \cap \Omega})\}
&& \text{in } \Omega,
\\ \label{eq:timeEvolWeight}
|\partial_t\vartheta + (B\cdot\nabla)\vartheta|(\cdot,t)
&\leq C|\vartheta|(\cdot,t)
&& \text{in } \Omega
\end{align}
for some constant~$C>0$ and all~$t \in [0,T]$. Finally, it holds
\begin{align}
\label{eq:motionByMCFCalibration}
|B\cdot\xi + \nabla\cdot\xi|(\cdot,t)
&\leq C\min\{1,\dist(\cdot,\overline{\p^*\mathscr{A}(t) \cap \Omega})\}
&& \text{in } \Omega
\end{align}
for some constant~$C>0$ and all~$t \in [0,T]$. 
\end{definition}

Given a $BV$~solution and a calibrated evolution with associated boundary adapted
gradient flow calibration, both being defined up to a common finite time horizon $T\in (0,\infty)$,
we then introduce a relative entropy functional 
\begin{align}
\label{eq:relEntropy}
\E_{\mathrm{relEn}}(A(t)|\mathscr{A}(t)) :=  c_0 \int_{\p^*A(t) \cap \Omega}
\big(1 - \nu_{A(t)} \cdot \xi(\cdot,t)\big) \,d\H^{d-1}, \quad t \in [0,T],
\end{align}
as well as a bulk error functional
\begin{align}
\label{eq:bulkError}
\E_{\mathrm{bulk}}(A(t)|\mathscr{A}(t)) := 
\int_{A(t) \Delta \mathscr{A}(t)} |\vartheta|(\cdot,t) \,dx, \quad t \in [0,T].
\end{align}
These two functionals turn out to be suitable measures for the difference
between the given $BV$~solution and the given calibrated evolution since 
they satisfy stability estimates in the following form.

\begin{theorem}
\label{theo:weakStrongUniqueness}
Let $d\geq 2$. Consider $T\in (0,\infty)$
and let $\mathscr{A}(t)$, $t\in [0,T]$, be a calibrated evolution
in the sense of Definition~\ref{def:calibratedEvolution}.
Let $A(t)$, $t\in [0,T^*]$ where $T^* \in (T,\infty)$, be
a distributional solution to MCF with constant contact angle~$\alpha$
in the sense of Definition~\ref{def:weak sol}.

For every associated boundary adapted gradient flow calibration~$(\xi,B,\vartheta)$,
there then exists a constant $C>0$ depending only on~$(\xi,B,\vartheta)$ such that
for almost every $T' \in (0,T)$
\begin{align}
\label{eq:stabilityEstimateRelEntropy}
\E_{\mathrm{relEn}}(A(T')|\mathscr{A}(T')) &\leq 
\E_{\mathrm{relEn}}(A(0)|\mathscr{A}(0))
+ C \int_{0}^{T'} \E_{\mathrm{relEn}}(A(t)|\mathscr{A}(t)) \,dt,
\\ \label{eq:stabilityEstimateBulkError}
\E_{\mathrm{bulk}}(A(T')|\mathscr{A}(T'))
&\leq \E_{\mathrm{bulk}}(A(0)|\mathscr{A}(0)) + 
\E_{\mathrm{relEn}}(A(0)|\mathscr{A}(0))
\\&~~~\nonumber
+ C \int_{0}^{T'} \E_{\mathrm{bulk}}(A(t)|\mathscr{A}(t))
+ \E_{\mathrm{relEn}}(A(t)|\mathscr{A}(t))\,dt.
\end{align}
In particular, we obtain weak-strong uniqueness in form of
\begin{align}
\nonumber
A(0) &= \mathscr{A}(0) \text{ up to a set of zero Lebesgue measure}
\\& \label{eq:weakStrongUniqueness}
\Rightarrow
A(t) = \mathscr{A}(t) \text{ up to a set of zero Lebesgue measure for a.e.\ } t \in (0,T).
\end{align}
\end{theorem}

The previous result is of course conditional in the sense that
we assume the existence of a calibrated evolution. To obtain
an unconditional statement in the sense of a usual weak-strong
uniqueness principle, one needs to show that sufficiently 
regular solutions to MCF with constant contact angle are calibrated.
This part of the story is worked out by Moser and the first author
in~\cite[Theorem~4]{Hensel2021c} for the planar setting $d=2$;
however, one strongly expects that the remaining physically relevant 
case of $d=3$ does not entail any additional conceptual difficulties.
In particular, the required construction is most probably substantially less subtle
than the triple line construction from~\cite{Hensel2021} for a double bubble cluster.
Anyway, in combination with our Theorem~\ref{theo:weakStrongUniqueness},
the existence result of \cite{Hensel2021c} entails in the planar case
weak-strong uniqueness and quantitative stability of $BV$ solutions
to~MCF with constant contact angle (with respect to a class of sufficiently
regular solutions, cf.\ \cite[Definition~10]{Hensel2021c} for details).
%

\section{Compactness and conditional convergence} \label{sec:conv}
This section is devoted to the proof of Theorem~\ref{theo:existence}.
The first subsection outlines the main steps of the proof and
formulates along the way the required intermediate results.
Proofs for these as well as Theorem~\ref{theo:existence} are postponed
to the second and third subsection, respectively.

\subsection{Main steps of the proof and intermediate results}
As a preliminary for this subsection, we start making precise what we mean
by a weak solution to~\eqref{eq:AllenCahn1}--\eqref{eq:AllenCahn3}.

\begin{definition}
\label{def:weakSolAllenCahn}
Let $T \in (0,\infty)$, $d \geq 2$, and consider an initial phase field $u_{\eps,0}$
with finite energy $E_\eps[u_{\eps,0}] < \infty$. Let the pair of nonlinearities $(W,\sigma)$
satisfy~\eqref{assump:W} and~\eqref{assump:Sigma1}.

A measurable function $u_\eps\colon\Omega \times [0,T) \to \R$ is henceforth
called a \emph{weak solution to~\emph{\eqref{eq:AllenCahn1}--\eqref{eq:AllenCahn3}}} if
\begin{enumerate}[1.]
			\item it satisfies $u_\eps \in H^1(0,T;L^2(\Omega)) 
						\cap L^\infty(0,T;H^1(\Omega) {\cap} L^4(\Omega))$,
			\item the initial data is attained in the sense of
						$u_\eps(\cdot,0) = u_{\eps,0}$ a.e.\ in $\Omega$,
			\item and finally for all $T' \in (0,T)$ and all test functions
						$\zeta \in C^\infty_c(\overline{\Omega} \times [0,T))$ it holds
						\begin{align}
						\label{eq:weakEvolEquation}
						\int_{0}^{T'} \int_{\Omega} \zeta \partial_tu_\eps \,dx dt
						= - \int_{0}^{T'} \int_{\Omega} \nabla\zeta \cdot \nabla u_\eps
						+ \zeta \frac{1}{\eps^2}W'(u_\eps) \,dx dt
						- \int_{0}^{T'} \int_{\p\Omega} \zeta \frac{1}{\eps}\sigma'(u_\eps) \,d\mathcal{H}^{d-1} dt.
						\end{align}
\end{enumerate}
\end{definition}

Existence of weak solutions and further properties of such (e.g., higher regularity) are
collected in~\cite[Lemma~6, Lemma~7 and Lemma~8]{Hensel2021c}.
We summarize these in the following result; for a proof, one may 
consult~\cite[Appendix~A]{Hensel2021c}.

\begin{lemma}
\label{lem:propertiesWeakSol}
In the setting of Definition~\ref{def:weakSolAllenCahn}, there always
exists a unique weak solution to~\emph{\eqref{eq:AllenCahn1}--\eqref{eq:AllenCahn3}}.
If we in addition assume that 
\begin{align}
\label{eq:LinftyBoundInitially}
|u_{\eps,0}| \leq 1 \quad \text{a.e.\ in } \Omega,
\end{align}
then the unique weak solution to~\emph{\eqref{eq:AllenCahn1}--\eqref{eq:AllenCahn3}}
is subject to the following additional properties:
\begin{enumerate}[1.]
\item (Uniform boundedness) The $L^\infty$-bound is conserved by the flow in the sense that
\begin{align}
\label{eq:LinftyBound}
|u_\eps(\cdot,T')| \leq 1 \quad \text{a.e.\ in } \Omega \text{ for all } T' \in [0,T).
\end{align}
\item (Higher regularity and interpretation of boundary condition) We have
\begin{align}
\label{eq:higherReg} 
 u_\eps \in L^2(0,T;H^2(\Omega))\cap C([0,T];H^1(\Omega))
\quad\text{and}\quad
\nabla\partial_t u_\eps \in L^2_{\mathrm{loc}}(0,T;L^2(\Omega)).
\end{align}
In particular, the PDE~\eqref{eq:AllenCahn1} is satisfied pointwise a.e.\ throughout $\Omega \times (0,T)$
and the nonlinear Robin boundary condition~\eqref{eq:AllenCahn3} holds in form of
\begin{align}
\label{eq:boundaryCond}
- \int_{\p\Omega} \zeta \frac{1}{\eps}\sigma'(u_\eps(\cdot,T')) \,d\mathcal{H}^{d-1}
= \int_{\Omega} \zeta\,\Delta u_\eps(\cdot,T') \,dx
+ \int_{\Omega} \nabla\zeta\cdot\nabla  u_\eps(\cdot,T') \,dx
\end{align}
for a.e.\ $T' \in (0,T)$ and all test functions $\zeta \in C^\infty(\overline{\Omega})$.
\item (Energy dissipation identity) For all $T' \in (0,T)$ it holds
\begin{align}
\label{eq:dissipIdentity}
E_\eps[ u_\eps(\cdot,T')] 
+ \int_{0}^{T'} \int_{\Omega} \eps|\partial_t  u_\eps|^2 \,dx dt
= E_\eps[ u_{\eps,0}].
\end{align}
\end{enumerate}
\end{lemma}

\subsubsection{Compactness}
With these preliminaries in place, the first step consists of
the extraction of an accumulation point of the sequence $(u_\eps)_{\eps > 0}$.
This is done along the lines of a standard compactness argument
which in turn is based on the well-known 
Modica--Mortola/Bogomol'nyi-trick~\cite{modicamortola, bogomolnyi}.
More precisely, recalling the definition~\eqref{def:psi}
of the function~$\psi$ one defines the map
\begin{align}
\label{def:approxIndicator}
\psi_\eps(x,t) := \psi(u_\eps(x,t)) = \int_{-1}^{u_\eps(x,t)}
\sqrt{2W(s')} \,ds',\quad (x,t) \in \Omega \times (0,T).
\end{align}
By the chain rule, $(\nabla,\partial_t)\psi_\eps = \sqrt{2 W(u_\eps)} (\nabla,\partial_t)u_\eps$
so that thanks to H\"{o}lder's inequality, the assumptions~\eqref{eq:uniformBoundInitialEnergy}
and~\eqref{eq:uniformLinftyBoundInitialData}, and the energy dissipation identity~\eqref{eq:dissipIdentity}
one obtains
\begin{align*}
\int_{0}^{T} \int_{\Omega} |(\nabla,\partial_t)\psi_\eps| \,dx dt
&\leq \sqrt{2} \bigg(\int_{0}^{T} \int_{\Omega} \frac{1}{\eps}W(u_\eps) \,dx dt\bigg)^\frac{1}{2}
\bigg(\int_{0}^{T} \int_{\Omega} \eps|(\nabla,\partial_t)u_\eps|^2 \,dx dt\bigg)^\frac{1}{2}
\\&
\lesssim_{T} \sup_{\eps > 0} E_\eps(u_{\eps,0}) < \infty. 
\end{align*}
The bound of the previous display entails by a standard compactness argument
the following convergence result~\eqref{eq:convergence}; for a proof, we refer to the
literature (e.g., the $\Gamma$-convergence result of Modica~\cite{Modica1987}). 
The lower semi-continuity statement~\eqref{eq:uniformBoundLimitEnergy}
follows from \cite[Proposition~1.2]{Modica1987} together with the already
mentioned Modica--Mortola/Bogomol'nyi-trick. The familiar $\frac{1}{2}$-H\"{o}lder continuity
of the volume of the evolving phase, see~\eqref{eq:contVolumePhases}, in turn 
follows from the following variant of the previous display
\begin{align*}
\int_{\Omega} |\psi_\eps(x,t) - \psi_\eps(x,s)| \,dx
&\leq \sqrt{2} \bigg(\int_{s}^{t} \int_{\Omega} \frac{1}{\eps}W(u_\eps) \,dx dt\bigg)^\frac{1}{2}
\bigg(\int_{0}^{T} \int_{\Omega} \eps|\partial_t u_\eps|^2 \,dx dt\bigg)^\frac{1}{2}
\\&
\leq \sqrt{2(t-s)}\sup_{\eps>0}E_\eps(u_{\eps,0}).
\end{align*}

\begin{lemma}
\label{lem:compactness}
In the setting of Theorem~\ref{theo:existence} i), there exists
a subsequence $\eps \downarrow 0$ and a one-parameter family
of sets of finite perimeter $A(t) \subset \Omega$, $t \in [0,T]$,
such that 
\begin{align}
\label{eq:convergence}
\psi_\eps &\to \psi_0 := \psi(2\chi_A{-}1) = c_0\chi_A
&&\text{strongly in } L^1(\Omega\times(0,T)) \text{ as } \eps \downarrow 0,
\\
\label{eq:convergence2}
\psi_\eps & \to \psi_0
&&\text{weakly* in } BV(\Omega\times(0,T)) \text{ as } \eps \downarrow 0,
\\ \label{eq:uniformBoundLimitEnergy}
\esssup_{t\in(0,T)} E(A(t)) &\leq E(A(0))
&&\text{as } \eps \downarrow 0,
\end{align}
where $\chi_A(x,t) := \chi_{A(t)}(x)$ for all $(x,t) \in \Omega \times [0,T]$
and $c_0$ is the surface tension constant~\eqref{def:surfaceTensionInterface}.
Moreover, after possibly modifying~$\chi_A$ on a null set of positive times $\mathcal{N}\subset(0,T)$
it holds \emph{for all} $0 \leq s,t \leq T$
\begin{align}
\label{eq:contVolumePhases}
c_0\int_{\Omega} |\chi_{A(t)} - \chi_{A(s)}| \,dx \leq \sqrt{2|t-s|}\sup_{\eps>0}E_\eps(u_{\eps,0}).
\end{align}
\end{lemma}

\subsubsection{Step~1 in the derivation of the motion law~\eqref{eq:motionLaw}: the curvature term}
We move on with discussing the first step towards the verification
of the desired motion law~\eqref{eq:motionLaw}. To this end, given a sufficiently regular test vector field 
$B \in C_c^1(\overline{\Omega} \times [0,T);\R^d)$ with $B \cdot \nu_{\p \Omega} =0$ 
on $\p\Omega \times (0,T)$, we test~\eqref{eq:weakEvolEquation} with $ \eps (B\cdot\nabla) u_\eps$
which indeed is an admissible test function due to
the regularity~\eqref{eq_regularityB} and~\eqref{eq:higherReg}:
\begin{align}
\nonumber
&- \int_{0}^{T} \int_{\Omega} \nabla \big(\eps (B\cdot\nabla) u_\eps\big) \cdot \nabla u_\eps
																+ \big(\eps (B\cdot\nabla) u_\eps\big)  \frac{1}{\eps^2}W'(u_\eps) \,dx dt
\\& \label{eq:aux1}
- \int_{0}^{T} \int_{\p\Omega} \big(\eps (B\cdot\nabla) u_\eps\big) 
\frac{1}{\eps}\sigma'(u_\eps) \,d\mathcal{H}^{d-1} dt
= \int_{0}^{T} \int_{\Omega} \big(\eps (B\cdot\nabla) u_\eps\big) \p_t u_\eps  \,dx dt.
\end{align}
We will now rewrite the three left hand side terms from the last display
in a form which resembles the desired structure of the two left hand side terms
of~\eqref{eq:motionLaw}. We remark that all of the subsequent computations are
justified thanks to the higher regularity~\eqref{eq:higherReg}. First,
by means of the chain rule in form of $\sigma'(u_\eps) (B\cdot\nabla)u_\eps
= (B\cdot\nabla)\sigma(u_\eps)$, an integration by parts on the boundary
of the container~$\Omega$ and recalling that $B$ is tangential along it,
cf.\ \eqref{eq:angleCondB}, we obtain
\begin{align*}
- \int_{0}^{T} \int_{\p\Omega} \big(\eps (B\cdot\nabla) u_\eps\big) 
\frac{1}{\eps}\sigma'(u_\eps) \,d\mathcal{H}^{d-1} dt
= \int_{0}^{T} \int_{\p\Omega} \sigma(u_\eps) 
(I_d {-} \nu_{\p\Omega}\otimes \nu_{\p \Omega}) : \nabla B \,d\mathcal{H}^{d-1} dt.
\end{align*}
Relying on the chain rule in form of $W'(u_\eps) (B\cdot\nabla)u_\eps
= (B\cdot\nabla)W(u_\eps)$ and integrating by parts also entails
\begin{align*}
- \int_{0}^{T} \int_{\Omega} \big(\eps (B\cdot\nabla) u_\eps\big)  \frac{1}{\eps^2}W'(u_\eps) \,dx dt
=  \int_{0}^{T} \int_{\Omega} \frac{1}{\eps}W(u_\eps) (\nabla \cdot B)  \,dx dt.
\end{align*}
Exploiting the symmetry of $\nabla^2 u_\eps$, integrating by parts,
and using again that~$B$ is tangential along~$\partial\Omega$ by~\eqref{eq:angleCondB} 
moreover shows
\begin{align*}
- \int_{0}^{T} \int_{\Omega} \eps \nabla u_\eps \otimes B : \nabla^2 u_\eps \,dx dt
= \int_{0}^{T} \int_{\Omega} \frac{\eps}{2} |\nabla u_\eps|^2 (\nabla \cdot B) \,dx dt,
\end{align*}
from which we in turn infer by the product rule
\begin{align*}
&- \int_{0}^{T} \int_{\Omega} \nabla \big(\eps (B\cdot\nabla) u_\eps\big) \cdot \nabla u_\eps \,dx dt
\\&
= \int_{0}^{T} \int_{\Omega} \frac{\eps}{2} |\nabla u_\eps|^2 (\nabla \cdot B) \,dx dt
- \int_{0}^{T} \int_{\Omega} \eps \nabla u_\eps \otimes \nabla u_\eps : \nabla B \,dx dt.
\end{align*}
In summary, we obtain from the previous displays the following version of~\eqref{eq:aux1}:
\begin{align}
\nonumber
&\int_{0}^{T} \int_{\Omega} \bigg(\Big(\frac{\eps}{2} |\nabla u_\eps|^2 + \frac{1}{\eps}W(u_\eps)\Big) I_d
- \eps \nabla u_\eps \otimes \nabla u_\eps\bigg) : \nabla B  \,dx dt
\\& \label{eq:aux2}
+ \int_{0}^{T} \int_{\p\Omega} \sigma(u_\eps) 
(I_d {-} \nu_{\p\Omega}\otimes \nu_{\p \Omega}) : \nabla B \,d\mathcal{H}^{d-1} dt
= \int_{0}^{T} \int_{\Omega} \big(\eps (B\cdot\nabla) u_\eps\big) \p_t u_\eps  \,dx dt.
\end{align}

In order to take the limit $\eps \downarrow 0$ on the left hand side of~\eqref{eq:aux2}, 
an inspection of the associated terms immediately shows that one requires at the very 
least information of the limiting behavior of the localized energies
\begin{align}
\label{def:locEnergyDensity}
E_\eps(u;\eta) &:= \int_{\Omega} \eta\Big(\frac{\eps}{2} |\nabla u|^2 + \frac{1}{\eps}W(u)\Big) \,dx
									+ \int_{\p\Omega} \eta\sigma(u) \,d\mathcal{H}^{d-1},
&& \eta \in C(\overline{\Omega}),
\\ \label{def:locaEnergyDensityLimit}
E(A;\eta) &:= c_0 \int_{\p^*A \cap \Omega} \eta \,d\mathcal{H}^{d-1}
+ c_0\cos\alpha \int_{\p^*A \cap \partial\Omega} \eta \,d\mathcal{H}^{d-1},
&& \eta \in C(\overline{\Omega}).
\end{align}
That the limit of~\eqref{def:locEnergyDensity} is given by the expected
quantity~\eqref{def:locaEnergyDensityLimit} is the content of the following result.

\begin{lemma}
\label{lem:convergenceLocEnergy}
Let the assumptions of Theorem~\ref{theo:existence} ii) be in place; in particular,
the energy convergence assumption~\eqref{assump:energyConvergence}
with respect to the map~$\chi_A$ and the subsequence $\eps \downarrow 0$ from 
the compactness Lemma~\ref{lem:compactness}. Then
\begin{align}
\label{eq:convergenceLocEnergy}
\lim_{\eps \downarrow 0} \int_{0}^{T} E_\eps\big(u_\eps(\cdot,t);\eta(\cdot,t)\big) \,dt
= \int_{0}^{T} E\big(A(t);\eta(\cdot,t)\big) \,dt 
\end{align}
for all $\eta \in C(\overline{\Omega}\times[0,T])$.
\end{lemma}

Hence, taking $\eta=\nabla \cdot B$ as a test function in the 
convergence~\eqref{eq:convergenceLocEnergy} and recalling~\eqref{eq:Young} shows
\begin{align}
\nonumber
&\lim_{\eps \downarrow 0} \bigg(\int_{0}^{T} \int_{\Omega} \Big(\frac{\eps}{2} |\nabla u_\eps|^2 
+ \frac{1}{\eps}W(u_\eps)\Big) I_d : \nabla B  \,dx dt
+ \int_{0}^{T} \int_{\p\Omega} \sigma(u_\eps) I_d : \nabla B \,d\mathcal{H}^{d-1} dt \bigg)
\\& \label{eq:convergenceCurvature1}
= c_0 \int_{0}^{T} \int_{\p^*A \cap \Omega} I_d : \nabla B \,d\mathcal{H}^{d-1} dt
+ \jump{\hat\sigma} \int_{0}^{T} \int_{\p^*A \cap \partial\Omega} I_d : \nabla B \,d\mathcal{H}^{d-1} dt.
\end{align}
Admittedly, this is only half of the story for recovering the curvature term on the 
left hand side of the desired motion law~\eqref{eq:motionLaw}. In order to take
the limit in those left hand side terms of~\eqref{eq:aux2} which are quadratic in 
the dependence on the ``normals'', we resort to a freezing argument which in turn
is facilitated by the introduction of the following localized relative entropy type
functionals:
\begin{align}
\label{def:locRelEntropy}
\E_\eps(u;\eta,\xi) &:= E_\eps(u;\eta) - \int_{\Omega} \eta(\xi\cdot\nabla)\big(\psi(u)\big) \,dx
									  - \int_{\p\Omega} \eta(\cos\alpha)\psi(u)  \,d\mathcal{H}^{d-1},
\\ \label{def:locRelEntropyLimit}
\E(A;\eta,\xi) &:= E(A;\eta) - c_0\int_{\p^*A \cap \Omega} \eta 
\xi\cdot\frac{\nabla\chi_A}{|\nabla\chi_A|} \,d\mathcal{H}^{d-1}
- c_0\cos\alpha\int_{\p^*A \cap \p\Omega} \eta \,d\mathcal{H}^{d-1}
\end{align}
for all test functions
\begin{align}
\label{eq:testFunctionLocRelEntropy}
\eta \in C(\overline{\Omega}) \quad\text{and}\quad \xi \in C(\overline{\Omega}) \text{ s.t.\ }
\xi \cdot \nu_{\p\Omega} = \cos\alpha \text{ along } \p\Omega. 
\end{align}
The two functionals~\eqref{def:locRelEntropy} and~\eqref{def:locRelEntropyLimit}
are particularly suited for a freezing argument in the normals due to following
two observations. First, each of the two functionals can be rewritten as
a perturbation of the corresponding localized energy functionals~\eqref{def:locEnergyDensity}
and~\eqref{def:locaEnergyDensityLimit}, respectively, which is linear --- and thus well-suited
to weak convergence methods --- in the dependence of the data $\psi(u)$ and $\chi_A$, respectively. 
Indeed, we may simply compute based on an integration by parts and the 
boundary condition~\eqref{eq:angleCondXi} of the test function~$\xi$
\begin{align}
\label{eq:altRepRelEntropy}
\E_\eps(u;\eta,\xi) &= E_\eps(u;\eta) + \int_{\Omega} 
\psi(u) \big(\eta(\nabla\cdot\xi) + (\xi\cdot\nabla)\eta\big) \,dx,
\\ \label{eq:altRepRelEntropyLimit}
\E(A;\eta,\xi) &= E(A;\eta) + \int_{A} 
\big(\eta(\nabla\cdot\xi) + (\xi\cdot\nabla)\eta \big) \,dx.
\end{align}
A direct consequence of these two identities is the following result.

\begin{lemma}
\label{lem:convergenceLocRelEntropy}
Let the assumptions of Theorem~\ref{theo:existence} ii) be in place; 
in particular, the energy convergence assumption~\eqref{assump:energyConvergence}
with respect to the map~$\chi_A$ and the subsequence $\eps \downarrow 0$ from 
the compactness Lemma~\ref{lem:compactness}. Then
\begin{align}
\label{eq:convergenceLocRelEntropy}
\lim_{\eps \downarrow 0} \int_{0}^{T} \E_\eps\big(u_\eps(\cdot,t);\eta(\cdot,t),\xi(\cdot,t)\big) \,dt
= \int_{0}^{T} \E\big(A(t);\eta(\cdot,t),\xi(\cdot,t)\big) \,dt 
\end{align}
for all $\eta \in C(\overline{\Omega}\times[0,T])$ and all $\xi \in C(\overline{\Omega}\times[0,T];\R^d)$
such that $\xi\cdot\nu_{\p\Omega} = \cos\alpha$ along $\partial\Omega \times (0,T)$.
\end{lemma}

Provided one requires in addition to~\eqref{eq:testFunctionLocRelEntropy} the global length constraint
\begin{align}
\label{eq:lengthConstraintTestField}
|\xi| \leq 1 \text{ on } \overline{\Omega},
\end{align} 
the second notable property of the two functionals~\eqref{def:locRelEntropy} and~\eqref{def:locRelEntropyLimit}
is that these are nonnegative quantities providing a tilt-excess type penalization of the difference in the
``normals'' $\xi$ and $\frac{\nabla(\psi(u))}{|\nabla(\psi(u))|}$, respectively $\xi$ 
and $\frac{\nabla\chi_A}{|\nabla\chi_A|}$. At the level of the sharp interface limit,
this is easily seen by simply plugging in~\eqref{def:locaEnergyDensityLimit}
into~\eqref{def:locRelEntropyLimit} as well as exploiting~\eqref{eq:lengthConstraintTestField} 
to the effect of
\begin{align}
\label{eq:tiltExcessControlLimit}
c_0\int_{\p^*A \cap \Omega} \eta\frac{1}{2}\bigg|\frac{\nabla\chi_A}{|\nabla\chi_A|} - \xi\bigg|^2 \,d\mathcal{H}^{d-1}
\leq c_0\int_{\p^*A \cap \Omega} \eta\bigg(1 - \xi\cdot\frac{\nabla\chi_A}{|\nabla\chi_A|}\bigg) \,d\mathcal{H}^{d-1}
= \E(A;\eta,\xi).
\end{align}
At the level of the phase field approximation, let $s\in\mathbb{S}^{d-1}$
be a fixed but arbitrary unit vector. We then define a unit vector field
\begin{align}
\label{eq:phaseFieldNormal}
\nu_\eps := 
\begin{cases}
\frac{\nabla u_\eps}{|\nabla u_\eps|} &\text{if } \nabla u_\eps \neq 0,
\\
s &\text{else}.
\end{cases}
\end{align}
Note that
\begin{align}
\label{eq:compPhaseFieldNormal}
\nu_\eps|\nabla u_\eps| = \nabla u_\eps
\quad\text{and}\quad \nu_\eps|\nabla \psi_\eps| = \nabla \psi_\eps.
\end{align}
One then obtains due to another application of the
Modica--Mortola/Bogomol'nyi-trick and the
shortness condition~\eqref{eq:lengthConstraintTestField}
for any finite energy phase field~$u$ subject to $|u| \leq 1$
\begin{align}
\nonumber
0 &\leq \int_{\Omega} \eta\frac{1}{2}\bigg(\sqrt{\eps}|\nabla u| - \frac{1}{\sqrt{\eps}}\sqrt{W(u)}\bigg)^2 \,dx
+ \int_{\Omega} \eta(1 - \xi\cdot\nu_\eps) |\nabla(\psi(u))| \,dx
\\& \label{eq:tiltExcessControl1}
= \int_{\Omega} \eta\Big(\frac{\eps}{2} |\nabla u|^2 + \frac{1}{\eps}W(u)\Big) \,dx
- \int_{\Omega} \eta(\xi\cdot\nabla)\big(\psi(u)\big) \,dx;
\end{align}
hence, together with the lower bound from assumption~\eqref{assump:Sigma2}
and recalling the definitions~\eqref{def:locEnergyDensity} and~\eqref{def:locRelEntropy}
\begin{align}
\label{eq:tiltExcessControl2}
0 &\leq \int_{\p\Omega} \eta\big(\sigma(u) - (\cos\alpha)\psi(u) \big) \,d\mathcal{H}^{d-1}
\leq \E_\eps(u;\eta,\xi).
\end{align}
In particular, for any finite energy phase field~$u$ satisfying $|u| \leq 1$ 
we deduce from the condition~\eqref{eq:lengthConstraintTestField}
and the two bounds~\eqref{eq:tiltExcessControl1} and~\eqref{eq:tiltExcessControl2}
\begin{align}
\nonumber
\int_{\Omega} \eta\frac{1}{2} |\nu_\eps - \xi|^2 |\nabla(\psi(u))| \,dx
&\leq  \int_{\Omega} \eta(1 - \xi\cdot\nu_\eps) |\nabla(\psi(u))| \,dx
\\& \label{eq:tiltExcessControl3}
\leq \E_\eps(u;\eta,\xi).
\end{align}
A straightforward argument along the lines of, e.g. \cite[Lemma~4]{Fischer2020},
also shows that for all finite energy phase fields~$u$ satisfying $|u| \leq 1$ 
\begin{align}
\label{eq:tiltExcessControl4}
\int_{\Omega} \eta\frac{1}{2} |\nu_\eps - \xi|^2 \eps|\nabla u|^2 \,dx
\lesssim \E_\eps(u;\eta,\xi).
\end{align}

Note that the two bounds~\eqref{eq:tiltExcessControl1} and~\eqref{eq:tiltExcessControl2}
simply mean that~$\E_\eps(u;\eta,\xi)$ controls the local lack of equipartition of the bulk and the boundary energy,
respectively (recall again that we needed the lower bound from assumption~\eqref{assump:Sigma2}
to provide a sign for the latter). Furthermore, since on one side we may optimize in the choices
of the test functions $\eta$ and $\xi$ in order to produce arbitrarily small values for~$\E(A;\eta,\xi)$,
and since on the other side the limit relative entropy~\eqref{eq:altRepRelEntropyLimit} controls
the asymptotic behaviour of its phase field version~\eqref{eq:altRepRelEntropy} 
by Lemma~\ref{lem:convergenceLocRelEntropy}, we infer from the two coercivity
estimates~\eqref{eq:tiltExcessControlLimit} and~\eqref{eq:tiltExcessControl3} that
the functionals~\eqref{def:locRelEntropy} and~\eqref{def:locRelEntropyLimit} are indeed
suitable candidates for the control of the error introduced by freezing the ``normals''.

As it turns out, the above ingredients are unfortunately still not 
sufficient to pass to the limit in the left hand side terms of~\eqref{eq:aux2}
and to identify this limit with the left hand side of~\eqref{eq:motionLaw}.
The last missing ingredient consists of post-processing
the assumed convergence of the total energies~\eqref{assump:energyConvergence}
to individual convergence of the bulk and boundary contributions,
respectively. This is precisely the point of the proof where
we now make use of the assumed upper bound from assumption~\eqref{assump:Sigma2}.

\begin{lemma}
\label{lem:convergenceTraces}
Let the assumptions of Theorem~\ref{theo:existence} ii) be in place; 
in particular, the energy convergence assumption~\eqref{assump:energyConvergence}
with respect to the map~$\chi_A$ and the subsequence $\eps \downarrow 0$ from 
the compactness Lemma~\ref{lem:compactness}. Then,
after possibly passing to another subsequence $\eps \downarrow 0$,
it holds for a.e.\ $t\in(0,T)$,
\begin{align}
\label{eq:strictConvergence}
\psi_\eps(\cdot,t) \to \psi_0(\cdot,t) 
= \psi(2\chi_A(\cdot,t){-}1) = c_0\chi_A(\cdot,t)
\quad\text{strictly in } BV(\Omega) \text{ as } \eps \downarrow 0. 
\end{align}
\end{lemma}

With all of the above ingredients in place, we are now
ready to establish the desired convergence of the quadratic terms.

\begin{proposition}
\label{prop:convergenceCurvatureTerm}
Let the assumptions of Theorem~\ref{theo:existence} ii) be in place; 
in particular, the energy convergence assumption~\eqref{assump:energyConvergence}
with respect to the map~$\chi_A$ and the subsequence $\eps \downarrow 0$ from 
the compactness Lemma~\ref{lem:compactness}. Then
\begin{align}
\nonumber
&\lim_{\eps \downarrow 0} \bigg(\int_{0}^{T} \int_{\Omega} 
\eps \nabla u_\eps \otimes \nabla u_\eps : \nabla B  \,dx dt
+ \int_{0}^{T} \int_{\p\Omega} \sigma(u_\eps) 
\nu_{\p\Omega} \otimes \nu_{\p\Omega} : \nabla B \,d\mathcal{H}^{d-1} dt \bigg)
\\&\label{eq:convergenceCurvature2}
= c_0 \int_{0}^{T} \int_{\p^*A(t) \cap \Omega} \nu_{A(t)} \otimes \nu_{A(t)} : \nabla B \,d\mathcal{H}^{d-1} dt
+ \jump{\hat\sigma} \int_{0}^{T} \int_{\p^*A(t) \cap \partial\Omega} 
\nu_{\p\Omega} \otimes \nu_{\p\Omega} : \nabla B \,d\mathcal{H}^{d-1} dt.
\end{align}
\end{proposition}
In summary, we obtain the left hand side of the motion law~\eqref{eq:motionLaw}
from the left hand side of its phase field approximation~\eqref{eq:aux2}
by means of the convergences~\eqref{eq:convergenceCurvature1} and~\eqref{eq:convergenceCurvature2}
for all test vector fields $B \in C_c^1(\overline{\Omega} \times [0,T);\R^d)$ with $B \cdot \nu_{\p \Omega} =0$ 
on $\p\Omega \times (0,T)$.

\subsubsection{Existence of a square-integrable normal velocity}
The next step of the proof of Theorem~\ref{theo:existence}
is concerned with the construction of a normal velocity for the
one-parameter family of interfaces $\p^*A(t) \cap \Omega$, $t \in (0,T)$.
This is done by showing that the measure~$\partial_t\chi_A$ is
absolutely continuous with respect to the product measure
$\mathcal{L}^1 \llcorner (0,T) \otimes (\mathcal{H}^{d-1} \llcorner (\p^*A(t) \cap \Omega))_{t\in(0,T)}$.
Once this is established, the desired normal velocity is then simply encoded in 
terms of the associated Radon--Nikod\'{y}m derivative. 

\begin{lemma}
\label{lem:constructionNormalVel}
Let the assumptions of Theorem~\ref{theo:existence} ii) be in place; 
in particular, the energy convergence assumption~\eqref{assump:energyConvergence}
with respect to the map~$\chi_A$ and the subsequence $\eps \downarrow 0$ from 
the compactness Lemma~\ref{lem:compactness}. Then, in the sense of finite Radon measures on 
$\Omega \times (0,T)$,
\begin{align}
\label{eq:absContinuityTimeDeriv}
\partial_t\chi_A \ll \mathcal{L}^1 \llcorner (0,T) \otimes 
(\mathcal{H}^{d-1} \llcorner (\p^*A(t) \cap \Omega))_{t\in(0,T)}.
\end{align}
Denoting the associated Radon--Nikod\'{y}m derivative by $V\colon\Omega\times (0,T)\to\R$,
it holds
\begin{align}
\label{eq:L2boundNormalSpeedOptimal}
c_0\int_{0}^{T'} \int_{\p^*A(t) \cap \Omega} V^2 \,d\H^{d-1} dt
\leq \liminf_{\eps \downarrow 0} \int_{0}^{T'} \int_{\Omega}
\eps|\partial_t u_\eps|^2 \,dx dt
\end{align}
for almost every $T' \in (0,T)$. Moreover, the evolution equation~\eqref{eq:evolSet}
for the one-parameter family of phases $A(t) \subset \Omega$, $t \in [0,T]$,
is satisfied.
\end{lemma}

\subsubsection{Step~2 in the derivation of the motion law~\eqref{eq:motionLaw}: the velocity term}
The final step in the proof of Theorem~\ref{theo:existence} 
consists of the limit passage in the right hand side term 
of the approximate motion law~\eqref{eq:aux2} and the identification
of the limit with the right hand side of~\eqref{eq:motionLaw}.
The necessary ingredients for this task are already provided by the previous
two paragraphs.

\begin{proposition}
\label{prop:convergenceVelocityTerm}
Let the assumptions of Theorem~\ref{theo:existence} ii) be in place; 
in particular, the energy convergence assumption~\eqref{assump:energyConvergence}
with respect to the map~$\chi_A$ and the subsequence $\eps \downarrow 0$ from 
the compactness Lemma~\ref{lem:compactness}. Then
\begin{align}
\label{eq:convergenceVelocityTerm}
\lim_{\eps \downarrow 0} \bigg(\int_{0}^{T} \int_{\Omega} 
\big(\eps (B\cdot\nabla) u_\eps\big) \p_t u_\eps  \,dx dt \bigg)
= \int_0^{T} \int_{\p^*A(t) \cap \Omega} B\cdot \nu_{A(t)} V \,d\H^{d-1} dt.
\end{align}
\end{proposition}
	
\subsection{Proofs for intermediate results} In this subsection, we provide
the proofs for the various intermediate results collected in the previous subsection.

\begin{proof}[Proof of Lemma~\ref{lem:convergenceLocEnergy}]
We remark that the following proof works without requiring 
the identity~\eqref{eq:relaxedBoundaryDensity}. In fact, we
will only use the inequality $\sigma \geq \hat \sigma$
which in turn follows immediately from the definition~\eqref{def:relaxedBoundaryDensity},
and the fact that $\hat\sigma\circ\psi^{-1}$ is $1$-Lipschitz,
see~\eqref{eq:LipschitzEnvelope}.

By linearity in $\eta$, it is enough to prove the statement for $\eta\in [0,1]$. 
Since by assumption the total energy (with $\eta=1$) converges,
upon replacing~$\eta$ by~$1{-}\eta$ it is sufficient to prove the lower bound
		\begin{align}
			\notag&\liminf_{\eps\downarrow0}\int_0^T\int_\Omega 
			\eta \left(\frac{\eps}{2} |\nabla u_\eps |^2 + \frac1\eps W(u_\eps) \right)\,dx dt 
			+ \int_0^T\int_{\p\Omega} \eta\, \sigma(u_\eps) \,d\H^{d-1} dt
			\\&\geq  c_0 \int_0^T \int_{\p^*A(t)\cap \Omega} \eta \,d\H^{d-1} dt
			+ \jump{\hat \sigma} \int_0^T \int_{\p^*A(t) \cap \p\Omega} \eta \, d\H^{d-1} dt. 
			\label{eq:energy lower bound}
		\end{align}
		To prove \eqref{eq:energy lower bound}, we start with Young's inequality 
		and the trivial inequality $\sigma \geq \hat \sigma$ to estimate from below
		\begin{align*}
			&\int_0^T\int_\Omega \eta \left(\frac{\eps}{2} |\nabla u_\eps |^2 + \frac1\eps W(u_\eps) \right)\,dx dt 
			+ \int_0^T\int_{\p\Omega} \eta\, \sigma(u_\eps) \,d\H^{d-1} dt
			\\&\geq \int_0^T\int_\Omega \eta \,\sqrt{2W(u_\eps)} |\nabla u_\eps| \,dx dt 
			+ \int_0^T\int_{\p\Omega} \eta\, \hat \sigma(u_\eps) \,d\H^{d-1} dt
			\\&=\int_0^T\int_\Omega \eta |\nabla \psi_\eps| \,dx dt 
				+ \int_0^T\int_{\p\Omega} \eta\, (\hat \sigma \circ \psi^{-1})(\psi_\eps) \,d\H^{d-1} dt,
		\end{align*}
		where $\psi_\eps = \psi \circ u_\eps$. 
		Now \eqref{eq:energy lower bound} follows from (a localized version of) 
		the lower semi-continuity statement~\cite[Proposition~1.2]{Modica1987} 
		since the function $\hat\sigma\circ \psi^{-1}$ is $1$-Lipschitz. 
		(Recall that we assumed w.l.o.g.\ $\hat \sigma(-1) =0$, so $\jump{\hat \sigma}=\hat\sigma(1)$.)
\end{proof}

\begin{proof}[Proof of Lemma~\ref{lem:convergenceLocRelEntropy}]
The asserted convergence~\eqref{eq:convergenceLocRelEntropy}
of the localized relative entropies is a direct consequence of the 
convergence~\eqref{eq:convergenceLocEnergy} of the localized energies,
the representations~\eqref{eq:altRepRelEntropy} and~\eqref{eq:altRepRelEntropyLimit},
and the compactness~\eqref{eq:convergence}.
\end{proof}
	
\begin{proof}[Proof of Lemma~\ref{lem:convergenceTraces}]
First, straightforward arguments allow to post-process
the assumed convergence~\eqref{assump:energyConvergence}
of the time-integrated energies to convergence of the individual energies;
at least after passing to another subsequence $\eps \downarrow 0$:
\begin{align}
\label{eq:individualEnergyConvergence}
\lim_{\eps \downarrow0} E_\eps(u_\eps(\,\cdot\,,t)) = E(A(t))
\quad\text{for a.e.\ } t \in (0,T).
\end{align}
For instance, one may couple the simple argument given in~\cite[Proof of Lemma~2.11, Step~1]{LauxSimon18}
with the $\Gamma$-convergence result of~\cite{Modica1987}.

Of course, by passing yet again to another subsequence $\eps \downarrow 0$, one may also
guarantee as a consequence of the compactness~\eqref{eq:convergence} that
\begin{align}
\label{eq:individualCompactness}
\psi_\eps(\cdot,t) &\to c_0\chi_A(\cdot,t)
\quad\text{strongly in } L^1(\Omega) \text{ as } \eps \downarrow 0
\text{ for a.e.\ } t \in (0,T).
\end{align}
For a proof of the claim~\eqref{eq:strictConvergence}, it thus suffices to establish
convergence of the total variations
\begin{align}
\label{eq:convergenceTV}
\int_{\Omega} |\nabla\psi_\eps(\cdot,t)| \,dx \to  
c_0\mathcal{H}^{d-1}(\p^*A(t)\cap\Omega) \quad\text{as } \eps \downarrow 0
\text{ for a.e.\ } t \in (0,T).
\end{align}
Since the set of times for which~\eqref{eq:convergenceTV} holds 
solely stems from~\eqref{eq:individualEnergyConvergence}
and~\eqref{eq:individualCompactness}, let us drop for the rest of
the argument the dependence on the time variable in the notation.

Making use of~\eqref{eq:individualEnergyConvergence} and the
Modica--Mortola/Bogomol'nyi-trick, we may estimate 
\begin{align*}
E(A) = \lim_{\eps \downarrow 0} E_\eps(u_\eps)
\geq \limsup_{\eps \downarrow 0} \bigg(\int_{\Omega} |\nabla\psi_\eps| \,dx
+ \int_{\p\Omega} \tau\circ\psi_\eps \,d\mathcal{H}^{d-1}\bigg)
\end{align*}
where we again set $\tau := \sigma\circ\psi^{-1}$. Adding zero in form of
$|\nabla\psi_\eps|=(1{-}\tau'(\psi_\eps))|\nabla\psi_\eps| + \tau'(\psi_\eps)|\nabla\psi_\eps|$,
relying on the chain rule in form of $\tau'(\psi_\eps)\nabla\psi_\eps=\nabla(\tau(\psi_\eps))$,
and using~$\tau'\geq 0$, which follows immediately from the monotonicity of~$\sigma$
in assumption~\eqref{assump:Sigma1},
we then get 
\begin{align*}
E(A) &\geq \limsup_{\eps \downarrow 0} \int_{\Omega} 
|\nabla(\psi_\eps - \tau\circ\psi_\eps)| \,dx
\\&~~~
+ \liminf_{\eps \downarrow 0} \bigg(\int_{\Omega} |\nabla(\tau\circ\psi_\eps)| \,dx
+ \int_{\p\Omega} \tau\circ\psi_\eps \,d\mathcal{H}^{d-1}\bigg).
\end{align*}
Since $\tau(0) = \sigma(\psi^{-1}(0)) = \sigma(-1) = 0$ due to~\eqref{def:psi}
and the first item of assumption~\eqref{assump:Sigma3}, we may extend $\psi_\eps$ by zero from~$\Omega$ 
to~$\R^d$ and thus identify the argument of the~$\liminf$ in the previous display as the
$BV$ seminorm~$|\cdot|_{BV(\R^d)}$ of $\tau\circ\psi_\eps$ in $\R^d$. 
In particular, by~\eqref{eq:individualCompactness},
continuity of~$\tau$, and lower semicontinuity of the $BV$ seminorm, we obtain
\begin{align*}
&\liminf_{\eps \downarrow 0} \bigg(\int_{\Omega} |\nabla\tau\circ\psi_\eps| \,dx
+ \int_{\p\Omega} \tau\circ\psi_\eps \,d\mathcal{H}^{d-1}\bigg)
= \liminf_{\eps \downarrow 0} |\tau\circ\psi_\eps|_{BV(\R^d)}
\geq |\tau(c_0\chi_{A})|_{BV(\R^d)}.
\end{align*}
Moreover, $\tau(c_0\chi_{A}) = \tau(c_0)\chi_{A}$ and,
by~\eqref{def:surfaceTensionInterface} and the second item 
of assumption~\eqref{assump:Sigma3}, $\tau(c_0) = \tau(\psi(1)) = \sigma(1)
= \jump{\hat\sigma}$. Hence, 
\begin{align*}
|\tau(c_0\chi_{A})|_{BV(\R^d)} &= 
\tau(c_0) \mathcal{H}^{d-1}(\p^*A \cap \Omega) + 
\jump{\hat\sigma} \mathcal{H}^{d-1}(\p^*A \cap \p\Omega)
\\&
= (\tau(c_0) {-} c_0) \mathcal{H}^{d-1}(\p^*A \cap \Omega) + E(A).
\end{align*}
The previous three displays therefore imply
\begin{align*}
E(A) &\geq \limsup_{\eps \downarrow 0} \int_{\Omega} 
|\nabla(\psi_\eps - \tau\circ\psi_\eps)| \,dx
+ (\tau(c_0) {-} c_0) \mathcal{H}^{d-1}(\p^*A \cap \Omega) + E(A).
\end{align*}
Rearranging terms and recalling $\psi_0=\psi(2\chi_A{-}1)=c_0\chi_A$,
which implies $\psi_0-\tau\circ\psi_0=(c_0-\tau(c_0))\chi_A$,
we obtain 
\begin{align*}
\limsup_{\eps \downarrow 0} \int_{\Omega} 
|\nabla(\psi_\eps - \tau\circ\psi_\eps)| \,dx
\leq \int_{\Omega} |\nabla(\psi_0 - \tau\circ\psi_0)| \,dx,
\end{align*}
so that due to~\eqref{eq:individualCompactness} we infer that
$\psi_\eps - \tau\circ\psi_\eps \to \psi_0 - \tau\circ\psi_0$
strictly in $BV(\Omega)$ as $\eps \downarrow 0$. 
The Fleming--Rishel coarea formula then implies that a.e.\ level-set of
$\psi_\eps - \tau\circ\psi_\eps$ converges strictly
to the corresponding level-set of $\psi_0 - \tau\circ\psi_0$. 
However, since $0<\tau'<1$ in $(0,c_0)$ as a consequence of the second item of
assumption~\eqref{assump:Sigma1} and the upper bound from assumption~\eqref{assump:Sigma2},
it follows that the map $[0,c_0] \in s' \mapsto s'{-}\tau(s')$ is bijective,
so that by~\eqref{eq:LinftyBound} in form of $\psi_\eps\in [0,c_0]$ we also
have that a.e.\ level-set of $\psi_{\eps}$ converges strictly
to the corresponding level-set of $c_0\chi_{A}$. 
Hence, the Fleming--Rishel coarea formula entails the claim~\eqref{eq:convergenceTV}
and thus~\eqref{eq:strictConvergence}.
\end{proof}

\begin{proof}[Proof of Proposition~\ref{prop:convergenceCurvatureTerm}]
We claim that for all $B \in C_c([0,T);C^2(\overline{\Omega};\R^d))$ with $B \cdot \nu_{\p \Omega} =0$ 
on $\p\Omega \times (0,T)$, all $\eta \in C([0,T];C^1(\overline{\Omega}))$, 
all $\xi \in C([0,T];C^1(\overline{\Omega};\R^d))$
with $|\xi|\leq 1$ and $\xi\cdot\nu_{\p\Omega} = \cos\alpha$ along $\partial\Omega \times (0,T)$,
and all $\delta \in (0,1)$ it holds
\begin{align}
\label{eq:aux3}
&\lim_{\eps \downarrow 0} \bigg|\bigg(\int_{0}^{T} \int_{\Omega} 
\eta\eps \nabla u_\eps \otimes \nabla u_\eps : \nabla B  \,dx dt
{+} \int_{0}^{T} \int_{\p\Omega} \eta\sigma(u_\eps) 
\nu_{\p\Omega} \otimes \nu_{\p\Omega} : \nabla B \,d\mathcal{H}^{d-1} dt \bigg)
\\& \nonumber
{-} \bigg(c_0 \int_{0}^{T} \int_{\p^*A(t) \cap \Omega} \eta\nu_{A(t)} \otimes \nu_{A(t)} : \nabla B \,d\mathcal{H}^{d-1} dt
{+} \jump{\hat\sigma} \int_{0}^{T} \int_{\p^*A(t) \cap \p\Omega} 
\eta \nu_{\p\Omega} \otimes \nu_{\p\Omega} : \nabla B \,d\mathcal{H}^{d-1} dt\bigg)\bigg|
\\& \nonumber
\lesssim_{T,\|\nabla B\|_{L^\infty},\|\eta\|_{L^\infty}}
\frac{1}{\delta}\E(A;|\eta|,\xi,T) + \delta E(A(0)),
\end{align}
where we abbreviated
\begin{align}
\label{eq:timeIntegratedRelEntropyLimit}
\E(A;\eta,\xi,T) :=
\int_{0}^{T} \E\big(A(t);\eta(\cdot,t),\xi(\cdot,t)\big) \,dt.
\end{align}
A localization argument (by means of a suitable partition of unity) 
together with a subsequent local optimization of the choice of the vector fields~$\xi$
along the lines of, e.g., \cite[Proof of Theorem~1.2]{LauxSimon18} shows that
the asserted convergence~\eqref{eq:convergenceCurvature2} is implied by the 
family of estimates~\eqref{eq:aux3}.

For a proof of~\eqref{eq:aux3}, 
we start estimating
\begin{align*}
&\bigg|\int_{0}^{T} \int_{\Omega} \eta\eps \nabla u_\eps \otimes \nabla u_\eps : \nabla B \,dx dt
- \int_{0}^{T} \int_{\Omega} \eta\xi\otimes\nu_\eps:\nabla B\,|\nabla\psi_\eps|\,dx dt\bigg|
\\&
\leq  \|\nabla B\|_{L^\infty}
\int_{0}^{T} \int_{\Omega} \Big|\eta\nu_\eps \otimes \nu_\eps \,\eps|\nabla u_\eps|^2
- \eta\xi\otimes\nu_\eps\,|\nabla\psi_\eps|\Big| \,dx dt
\\&
\leq  \|\nabla B\|_{L^\infty} \int_{0}^{T} \int_{\Omega} 
|\eta|\Big|\eps|\nabla u_\eps|^2 - |\nabla\psi_\eps|\Big| \,dx dt
+  \|\nabla B\|_{L^\infty} \int_{0}^{T} \int_{\Omega} 
|\eta||\nu_\eps - \xi| |\nabla\psi_\eps| \,dx dt.
\end{align*}
By H\"{o}lder's and Young's inequality as well as the coercivity estimate~\eqref{eq:tiltExcessControl3}
\begin{align*}
\int_{0}^{T} \int_{\Omega} 
|\eta||\nu_\eps - \xi| |\nabla\psi_\eps| \,dx dt
\lesssim \delta T\|\eta\|_{L^\infty}\sup_{t\in(0,T)} \int_{\Omega} |\nabla\psi_\eps(\cdot,t)| \,dx
+ \frac{1}{\delta}\E_\eps(u_\eps;|\eta|,\xi,T),
\end{align*}
where the abbreviation~$\E_\eps(u_\eps;\eta,\xi,T)$ is defined analogously
to~\eqref{eq:timeIntegratedRelEntropyLimit}. Furthermore, since $\nabla\psi_\eps
=\sqrt{2W(u_\eps)}\nabla u_\eps$ we get
$\big|\eps|\nabla u_\eps|^2 - |\nabla\psi_\eps|\big| = \sqrt{\eps}|\nabla u_\eps|
\big(\sqrt{\eps}|\nabla u_\eps| - \frac{1}{\sqrt{\eps}}\sqrt{2W(u_\eps)}\big)$,
hence by another application of H\"{o}lder's and Young's inequality
and this time the coercivity estimate~\eqref{eq:tiltExcessControl3}
\begin{align*}
\int_{0}^{T} \int_{\Omega} 
|\eta|\Big|\eps|\nabla u_\eps|^2 - |\nabla\psi_\eps|\Big| \,dx dt
\lesssim \delta T\|\eta\|_{L^\infty}\sup_{t\in(0,T)} 
\int_{\Omega} \eps|\nabla u_\eps(\cdot,t)|^2 \,dx
+ \frac{1}{\delta}\E_\eps(u_\eps;|\eta|,\xi,T).
\end{align*}
Due to $\int_{\Omega} |\nabla\psi_\eps(\cdot,t)| \,dx \leq E_\eps(u_\eps(\cdot,t))$
and $\int_{\Omega} \eps|\nabla u_\eps(\cdot,t)|^2 \,dx \leq 2E_\eps(u_\eps(\cdot,t))$,
we may post-process the previous three displays based on~\eqref{eq:dissipIdentity},
\eqref{eq:initialEnergyConvergence} and~\eqref{eq:convergenceLocRelEntropy} to the effect of
\begin{align}
\nonumber
&\bigg|\int_{0}^{T} \int_{\Omega} \eta\eps \nabla u_\eps \otimes \nabla u_\eps : \nabla B \,dx dt
- \int_{0}^{T} \int_{\Omega} \eta\xi\otimes\nu_\eps:\nabla B\,|\nabla\psi_\eps|\,dx dt\bigg|
\\& \label{eq:aux4}
\lesssim_{T,\|\nabla B\|_{L^\infty},\|\eta\|_{L^\infty}}
\frac{1}{\delta}\E(A;|\eta|,\xi,T) + \delta E(A(0)).
\end{align} 
A similar but even simpler argument based on~\eqref{eq:uniformBoundLimitEnergy}
and~\eqref{eq:tiltExcessControlLimit} also shows
\begin{align}
\nonumber
&\bigg|\int_{0}^{T} \int_{\p^*A(t) \cap \Omega} \eta\nu_{A(t)} \otimes \nu_{A(t)} : \nabla B \,d\mathcal{H}^{d-1} dt
- \int_{0}^{T} \int_{\p^*A(t) \cap \Omega} \eta\xi \otimes \nu_{A(t)} : \nabla B \,d\mathcal{H}^{d-1} dt\bigg|
\\& \label{eq:aux5}
\lesssim_{T,\|\nabla B\|_{L^\infty},\|\eta\|_{L^\infty}}
\frac{1}{\delta}\E(A;|\eta|,\xi,T) + \delta E(A(0)).
\end{align}

Next, by the second item of~\eqref{eq:compPhaseFieldNormal}, an integration by parts,
decomposing~$\xi$ into tangential and normal components, and making use
of the boundary condition~$\xi\cdot\nu_{\p\Omega}=\cos\alpha$ we obtain
\begin{align*}
&\int_{0}^{T} \int_{\Omega} \eta\xi\otimes\nu_\eps:\nabla B\,|\nabla\psi_\eps|\,dx dt
+ \int_{0}^{T} \int_{\p\Omega} \eta\sigma(u_\eps) 
\nu_{\p\Omega} \otimes \nu_{\p\Omega} : \nabla B \,d\mathcal{H}^{d-1} dt
\\&
= \int_{0}^{T} \int_{\Omega} \eta\xi\otimes\nabla\psi_\eps:\nabla B \,dx dt
+ \int_{0}^{T} \int_{\p\Omega} \eta\sigma(u_\eps) 
\nu_{\p\Omega} \otimes \nu_{\p\Omega} : \nabla B \,d\mathcal{H}^{d-1} dt
\\&
= - \int_{0}^{T} \int_{\p\Omega} \psi_\eps\eta\xi\otimes \nu_{\p\Omega}:\nabla B \,d\mathcal{H}^{d-1} dt
+ \int_{0}^{T} \int_{\p\Omega} \eta\sigma(u_\eps) 
\nu_{\p\Omega} \otimes \nu_{\p\Omega} : \nabla B \,d\mathcal{H}^{d-1} dt
\\&~~~
- \int_{0}^{T} \int_{\Omega} \psi_\eps\nabla\cdot\big(\eta(\nabla B)^\mathsf{T}\xi\big) \,dx dt
\\&
= \int_{0}^{T} \int_{\p\Omega} \eta\big(\sigma(u_\eps) - \psi_\eps\cos\alpha\big)
\nu_{\p\Omega} \otimes \nu_{\p\Omega} : \nabla B \,d\mathcal{H}^{d-1} dt
\\&~~~
- \int_{0}^{T} \int_{\p\Omega} \eta\psi_\eps(I_d{-}\nu_{\p\Omega}\otimes\nu_{\p\Omega})
\xi\otimes \nu_{\p\Omega}:\nabla B \,d\mathcal{H}^{d-1} dt
- \int_{0}^{T} \int_{\Omega} \psi_\eps\nabla\cdot\big(\eta(\nabla B)^\mathsf{T}\xi\big) \,dx dt.
\end{align*}
Recall now from standard $BV$ theory that strict convergence in $BV(\Omega)$
implies convergence of the (well-defined) traces in $L^1(\p\Omega,\mathcal{H}^{d-1}\llcorner\p\Omega)$;
see, e.g., \cite[Theorem~3.88]{AmbrosioFuscoPallara}. Hence, by means of the 
convergences~\eqref{eq:convergence} and~\eqref{eq:strictConvergence} as
well as the second item of assumption~\eqref{assump:Sigma3} we deduce
\begin{align*}
&\lim_{\eps \downarrow 0} \bigg(\int_{0}^{T} \int_{\Omega} 
\eta\xi\otimes\nu_\eps:\nabla B\,|\nabla\psi_\eps|\,dx dt
+ \int_{0}^{T} \int_{\p\Omega} \eta\sigma(u_\eps) 
\nu_{\p\Omega} \otimes \nu_{\p\Omega} : \nabla B \,d\mathcal{H}^{d-1} dt\bigg)
\\&
= - c_0\int_{0}^{T} \int_{\p^*A(t) \cap \p\Omega} \eta(I_d{-}\nu_{\p\Omega}\otimes\nu_{\p\Omega})
\xi\otimes \nu_{\p\Omega}:\nabla B \,d\mathcal{H}^{d-1} dt
\\&~~~
- c_0\int_{0}^{T} \int_{\Omega} \chi_{A(t)}\nabla\cdot\big(\eta(\nabla B)^\mathsf{T}\xi\big) \,dx dt.
\end{align*}
A further integration by parts in combination with
the boundary condition~$\xi\cdot\nu_{\p\Omega}=\cos\alpha$ 
upgrades the previous display to
\begin{align*}
&\lim_{\eps \downarrow 0} \bigg(\int_{0}^{T} \int_{\Omega} 
\eta\xi\otimes\nu_\eps:\nabla B\,|\nabla\psi_\eps|\,dx dt
+ \int_{0}^{T} \int_{\p\Omega} \eta\sigma(u_\eps) 
\nu_{\p\Omega} \otimes \nu_{\p\Omega} : \nabla B \,d\mathcal{H}^{d-1} dt\bigg)
\\&
= c_0\int_{0}^{T} \int_{\p^*A(t) \cap \Omega} \eta\xi \otimes \nu_{A(t)} : \nabla B \,d\mathcal{H}^{d-1} dt
+ c_0\int_{0}^{T} \int_{\p^*A(t) \cap \p\Omega} \eta\xi \otimes \nu_{\p\Omega} : \nabla B \,d\mathcal{H}^{d-1} dt
\\&~~~
- c_0\int_{0}^{T} \int_{\p^*A(t) \cap \p\Omega} \eta(I_d{-}\nu_{\p\Omega}\otimes\nu_{\p\Omega})
\xi\otimes \nu_{\p\Omega}:\nabla B \,d\mathcal{H}^{d-1} dt
\\&
= c_0\int_{0}^{T} \int_{\p^*A(t) \cap \Omega} \eta\xi \otimes \nu_{A(t)} : \nabla B \,d\mathcal{H}^{d-1} dt
+ \jump{\hat\sigma} \int_{0}^{T} \int_{\p^*A(t) \cap \p\Omega} 
\eta \nu_{\p\Omega} \otimes \nu_{\p\Omega} : \nabla B \,d\mathcal{H}^{d-1} dt.
\end{align*}
Hence, the desired estimate~\eqref{eq:aux3} follows from the previous display
and the estimates~\eqref{eq:aux4} and~\eqref{eq:aux5}.
\end{proof}

\begin{proof}[Proof of Lemma~\ref{lem:constructionNormalVel}]
Let $U\subset\Omega$ and $V\subset(0,T)$ be two open sets,
and consider $\zeta \in C^\infty_c(\Omega\times(0,T))$
such that $|\zeta|\leq 1$ and $\supp\zeta\subset U\times V$.
We then estimate exploiting the compactness~\eqref{eq:convergence},
the chain rule in form of $\partial_t\psi_\eps=\sqrt{2W(u_\eps)}\partial_t u_\eps$,
H\"{o}lder's inequality, the energy dissipation identity~\eqref{eq:dissipIdentity}
together with the assumption~\eqref{eq:initialEnergyConvergence}, the convergence~\eqref{eq:convergenceLocEnergy},
the convergence of the trace of~$\psi_\eps$ to the trace of $c_0\chi_A$ 
due to~\eqref{eq:strictConvergence}, the continuity of $\tau:=\sigma\circ\psi^{-1}$,
and finally the identity $\tau(c_0)=\sigma(1)=\jump{\hat\sigma}$ following from the second
item of assumption~\eqref{assump:Sigma3}
\begin{align*}
c_0\langle\partial_t\chi_A,\zeta\rangle
&=\lim_{\eps \downarrow 0} \bigg(- \int_{0}^T\int_{\Omega} \psi_\eps\partial_t\zeta \,dx dt\bigg)
\\&
\leq \bigg(\liminf_{\eps \downarrow 0}\int_{0}^{T}\int_{\Omega}\eps|\partial_tu_\eps|^2\,dx dt\bigg)^\frac{1}{2}
\bigg(\limsup_{\eps \downarrow 0}\int_{0}^{T}\int_{\Omega}|\zeta|^2\frac{2}{\eps}W(u_\eps)\,dx dt\bigg)^\frac{1}{2}
\\&
\leq \sqrt{2} E^\frac{1}{2}(\chi_{A(0)})
\bigg(\limsup_{\eps \downarrow 0}\int_{0}^{T} E_\eps(u_\eps(\cdot,t);|\zeta(\cdot,t)|^2) 
- \Big(\int_{\p\Omega} |\zeta|^2\tau(\psi_\eps) \,d\mathcal{H}^{d-1}\Big) \,dt\bigg)^\frac{1}{2}
\\&
= \sqrt{2} E^\frac{1}{2}(\chi_{A(0)}) 
\bigg(c_0\int_{0}^T\int_{\p^*A(t)\cap \Omega} 
|\zeta|^2\,d\mathcal{H}^{d-1} dt\bigg)^\frac{1}{2}.
\end{align*}
In other words, $|\partial_t\chi_A|(U \times V) \lesssim 
\big|\mathcal{L}^1 \llcorner (0,T) \otimes 
(\mathcal{H}^{d-1} \llcorner (\p^*A(t) \cap \Omega))_{t\in(0,T)}\big|^\frac{1}{2}(U\times V)$
from which the claim~\eqref{eq:absContinuityTimeDeriv} immediately follows. Note however that the
$L^2$-estimate~\eqref{eq:L2boundNormalSpeedOptimal} for the associated Radon--Nikod\'{y}m derivative~$V$
does not follow from the previous estimate since the latter is suboptimal by a factor of~$\sqrt{2}$.
However, note that by the Modica--Mortola/Bogomol'nyi-trick and the above arguments
\begin{align*}
c_0\int_{0}^{T} \int_{\p^*A(t) \cap \Omega} |\zeta|^2\,d\mathcal{H}^{d-1} dt
&= \lim_{\eps \downarrow 0} \int_{0}^{T} E_\eps(u_\eps(\cdot,t);|\zeta(\cdot,t)|^2) 
- \Big(\int_{\p\Omega} |\zeta|^2\tau(\psi_\eps) \,d\mathcal{H}^{d-1}\Big) \,dt
\\&
\geq \limsup_{\eps \downarrow 0} \int_{0}^{T} \int_{\Omega} |\zeta|^2|\nabla\psi_\eps| \,dx dt.
\end{align*}
Hence, 
\begin{align*}
c_0\int_{0}^{T} \int_{\p^*A(t) \cap \Omega} |\zeta|^2\,d\mathcal{H}^{d-1} dt
&= \lim_{\eps \downarrow 0} \int_{0}^{T} \int_{\Omega} |\zeta|^2|\nabla\psi_\eps| \,dx dt,
\end{align*}
and the argument in favor of~\cite[Lemma~2.11]{LauxSimon18} ensures
\begin{align*}
\lim_{\eps \downarrow 0}\int_{0}^{T}\int_{\Omega}|\zeta|^2\frac{2}{\eps}W(u_\eps)\,dx dt
= \lim_{\eps \downarrow 0} \int_{0}^{T} E_\eps(u_\eps(\cdot,t);|\zeta(\cdot,t)|^2) 
- \Big(\int_{\p\Omega} |\zeta|^2\tau(\psi_\eps) \,d\mathcal{H}^{d-1}\Big) \,dt.
\end{align*}
This in turn allows to estimate in an optimal fashion
\begin{align*}
c_0\int_{0}^T\int_{\p^*A(t)\cap \Omega} V\zeta\,d\mathcal{H}^{d-1} dt
\leq \bigg(\liminf_{\eps \downarrow 0}\int_{0}^{T}\int_{\Omega}\eps|\partial_tu_\eps|^2\,dx dt\bigg)^\frac{1}{2}
\bigg(c_0\int_{0}^T\int_{\p^*A(t)\cap \Omega} |\zeta|^2\,d\mathcal{H}^{d-1} dt\bigg)^\frac{1}{2}
\end{align*}
which implies the $L^2$-estimate~\eqref{eq:L2boundNormalSpeedOptimal}.

Finally, the evolution equation~\eqref{eq:evolSet} is an immediate
consequence of the very definition of the Radon--Nikod\'{y}m derivative~$V$;
at least for compactly supported and smooth test functions $\zeta\in C^\infty_c(\Omega\times (0,T))$.
Since $\chi_A\in C([0,T];L^1(\Omega))$, straightforward approximation arguments allow to lift this to
the required class of test functions $\zeta \in C_c^\infty (\overline{\Omega} \times [0,T))$.
\end{proof}

\begin{proof}[Proof of Proposition~\ref{prop:convergenceVelocityTerm}]
Analogous to the proof of Proposition~\ref{prop:convergenceCurvatureTerm},
it suffices to show for all $B \in C_c^1([0,T);C(\overline{\Omega};\R^d))$ with $B \cdot \nu_{\p \Omega} =0$ 
on $\p\Omega \times (0,T)$, all $\eta \in C^1([0,T];C(\overline{\Omega}))$, 
all $\xi \in C^1([0,T];C(\overline{\Omega};\R^d))$
with $|\xi|\leq 1$ and $\xi\cdot\nu_{\p\Omega} = \cos\alpha$ along $\partial\Omega \times (0,T)$,
and all $\delta \in (0,1)$ that
\begin{align}
\nonumber
&\lim_{\eps \downarrow 0} \bigg|\int_{0}^{T} \int_{\Omega} 
\eta\big(\eps (B\cdot\nabla) u_\eps\big) \p_t u_\eps  \,dx dt 
- \int_0^{T} \int_{\p^*A(t) \cap \Omega} \eta B\cdot \nu_{A(t)} V \,d\H^{d-1} dt\bigg|
\\& \label{eq:aux6}
\lesssim_{T,\|(B,\,\xi,\,\eta)\|_{L^\infty}}
\frac{1}{\delta}\E(A;|\eta|,\xi,T) + \delta E(A(0)).
\end{align}

For an argument in favor of~\eqref{eq:aux6}, we first rewrite
by recalling~\eqref{eq:phaseFieldNormal} and~\eqref{eq:compPhaseFieldNormal},
adding zero twice, and exploiting the chain rule in form of 
$\partial_t\psi_\eps=\sqrt{2W(u_\eps)}\partial_t u_\eps$
\begin{align*}
\int_{0}^{T} \int_{\Omega} \eta\big(\eps (B\cdot\nabla) u_\eps\big) \p_t u_\eps  \,dx dt 
&= \int_{0}^{T} \int_{\Omega} \eta B\cdot(\nu_\eps - \xi) \sqrt{\eps} |\nabla u_\eps|
\, \sqrt{\eps}\p_t u_\eps  \,dx dt 
\\&~~~
+ \int_{0}^{T} \int_{\Omega} \eta (B\cdot\xi)
\bigg(\sqrt{\eps} |\nabla u_\eps| - \frac{\sqrt{2W(u_\eps)}}{\sqrt{\eps}}\bigg) 
\, \sqrt{\eps}\p_t u_\eps  \,dx dt 
\\&~~~
+ \int_{0}^{T} \int_{\Omega} \eta (B\cdot\xi) \partial_t\psi_\eps \,dx dt.
\end{align*}
Integrating by parts, taking limits based on~\eqref{eq:convergence}
and~\eqref{eq:convergenceInitialCondition}, and plugging
in $\eta(B\cdot\xi)$ as a test function into the
(already established) evolution equation~\eqref{eq:evolSet} moreover yields
\begin{align*}
&\lim_{\eps \downarrow 0} 
\int_{0}^{T} \int_{\Omega} \eta (B\cdot\xi) \partial_t\psi_\eps \,dx dt
\\&
= \lim_{\eps \downarrow 0} \bigg(- \int_{0}^{T} \int_{\Omega} \psi_\eps\partial_t\big(\eta (B\cdot\xi)\big) \,dx dt
- \int_{\Omega} \psi_\eps(\cdot,0) \big(\eta (B\cdot\xi)\big)(\cdot,0) \,dx\bigg)
\\&
= \int_0^{T} \int_{\p^*A(t)\cap \Omega} \eta (B\cdot\xi) V \,d\H^{d-1} dt.
\end{align*}
In view of~\eqref{eq:dissipIdentity}, \eqref{eq:initialEnergyConvergence}
and \eqref{eq:tiltExcessControl4}, the arguments from the proof of 
Proposition~\ref{prop:convergenceCurvatureTerm} together with the previous
two displays ensure
\begin{align}
\nonumber
&\lim_{\eps \downarrow 0} \bigg|\int_{0}^{T} \int_{\Omega} 
\eta\big(\eps (B\cdot\nabla) u_\eps\big) \p_t u_\eps  \,dx dt 
- \int_0^{T} \int_{\p^*A(t) \cap \Omega} \eta B\cdot \xi V \,d\H^{d-1} dt\bigg|
\\& \label{eq:aux7}
\lesssim_{T,\|(B,\,\xi,\,\eta)\|_{L^\infty}}
\frac{1}{\delta}\E(A;|\eta|,\xi,T) + \delta E(A(0)).
\end{align}
Finally, the $L^2$-estimate~\eqref{eq:L2boundNormalSpeedOptimal} for~$V$
together with~\eqref{eq:tiltExcessControlLimit} implies
\begin{align*}
\lim_{\eps \downarrow 0} \bigg|\int_0^{T} \int_{\p^*A(t) \cap \Omega} 
\eta B\cdot (\xi - \nu_{A(t)}) V \,d\H^{d-1} dt\bigg|
\lesssim_{T,\|(B,\,\xi,\,\eta)\|_{L^\infty}}
\frac{1}{\delta}\E(A;|\eta|,\xi,T) + \delta E(A(0)).
\end{align*} 
The previous display upgrades~\eqref{eq:aux7} to~\eqref{eq:aux6},
and thus concludes the proof.
\end{proof}

\subsection{Proof of Theorem~\ref{theo:existence}}
The first part of Theorem~\ref{theo:existence} (i.e., the
compactness claim) is already contained in Lemma~\ref{lem:compactness}.
The assertions from the first item of Definition~\ref{def:weak sol}
(Existence of normal velocity) are a consequence of Lemma~\ref{lem:constructionNormalVel}.
The combination of~\eqref{eq:convergenceCurvature1}, Proposition~\ref{prop:convergenceCurvatureTerm}
and Proposition~\ref{prop:convergenceVelocityTerm} entail
the second item of Definition~\ref{def:weak sol} (Motion law).
Finally, the optimal energy dissipation rate from the third 
item of Definition~\ref{def:weak sol} follows from~\eqref{eq:dissipIdentity},
\eqref{eq:individualEnergyConvergence}, \eqref{eq:L2boundNormalSpeedOptimal}
and~\eqref{eq:initialEnergyConvergence}. \qed

\section{Uniqueness properties of $BV$ solutions to MCF with constant contact angle}	\label{sec:uniqueness}
In this section, we turn to the question of weak-strong uniqueness for
distributional solutions to MCF with constant contact angle in the 
sense of Definition~\ref{def:weak sol}. To this end, we split
the proof of Theorem~\ref{theo:weakStrongUniqueness} into three steps.
The first two deal with the derivation of a suitable estimate for the time
evolution of the relative entropy and the bulk error, respectively.
We remark that the derivation of the former solely relies 
on the boundary conditions~\eqref{eq:angleCondXi}--\eqref{eq:angleCondB}
for the pair of vector fields~$(\xi,B)$, whereas the derivation of the latter only
makes use of the sign conditions~\eqref{eq:signWeightExterior}--\eqref{eq:signWeightInterface} 
imposed on the weight~$\vartheta$. The third step post-processes these estimates 
to~\eqref{eq:stabilityEstimateRelEntropy}, \eqref{eq:stabilityEstimateBulkError} 
and~\eqref{eq:weakStrongUniqueness} by means of the remaining properties of a 
boundary adapted gradient flow calibration.

\subsection{Quantitative stability of the relative entropy}
In the setting of Theorem~\ref{theo:weakStrongUniqueness},
we claim that for a.e.\ $T'\in [0,T]$
\begin{align}
\nonumber
&\E_{\mathrm{relEn}}(A(T')|\mathscr{A}(T')) + 
 c_0 \int_{0}^{T'} \int_{\p^*A(t) \cap \Omega}
\frac{1}{2} |V {-} (\nabla\cdot\xi)|^2 \,d\H^{d-1} dt
\\&\nonumber
+  c_0 \int_{0}^{T'} \int_{\p^*A(t) \cap \Omega}
\frac{1}{2} |V{+}(B\cdot\xi)|^2 \,d\H^{d-1} dt
\\& \label{eq:timeEvolRelEntropy}
\leq \E_{\mathrm{relEn}}(A(0)|\mathscr{A}(0)) + 
 c_0 \int_{0}^{T'} \int_{\p^*A(t) \cap \Omega}
\frac{1}{2} |(B\cdot\xi) {+}(\nabla\cdot\xi)|^2 \,d\H^{d-1} dt
\\&~~~ \nonumber
-  c_0 \int_{0}^{T'} \int_{\p^*A(t) \cap \Omega} 
B\cdot(\nu_{A(t)}{-}\xi)(V{-}(\nabla\cdot\xi)) \,d\H^{d-1} dt
\\&~~~ \nonumber
-  c_0 \int_{0}^{T'} \int_{\p^*A(t) \cap \Omega}
(\nu_{A(t)} {-} \xi) \cdot (\p_t\xi {+} (B\cdot\nabla)\xi {+} (\nabla B)^\mathsf{T}\xi) \,d\H^{d-1} dt
\\&~~~ \nonumber
-  c_0 \int_{0}^{T'} \int_{\p^*A(t) \cap \Omega}
\xi \cdot (\p_t\xi {+} (B\cdot\nabla)\xi) \,d\H^{d-1} dt
\\&~~~ \nonumber
-  c_0 \int_{0}^{T'} \int_{\p^*A(t) \cap \Omega}
(\nu_{A(t)}\cdot\xi - 1)(\nabla \cdot B) \,d\H^{d-1} dt
\\&~~~ \nonumber
-  c_0 \int_{0}^{T'} \int_{\p^*A(t) \cap \Omega}
(\nu_{A(t)} {-} \xi) \otimes (\nu_{A(t)} {-} \xi) : \nabla B \,d\H^{d-1} dt.
\end{align}

\begin{proof}[Proof of~\eqref{eq:timeEvolRelEntropy}]
It essentially follows from Subsection~2.3.3 of the PhD~thesis~\cite{Hensel2021a}
of the first author that
\begin{align}
\nonumber
&\E_{\mathrm{relEn}}(A(T')|\mathscr{A}(T')) - \E_{\mathrm{relEn}}(A(0)|\mathscr{A}(0))
\\& \label{eq:timeEvolRelEntropyAux0}
= E(A(T')) - E(A(0))
+  c_0 \int_{0}^{T'} \int_{\p^*A(t) \cap \Omega}
(\nabla\cdot\xi) (V + B\cdot\nu_{A(t)}) \,d\H^{d-1} dt
\\&~~~ \nonumber
-  c_0 \int_{0}^{T'} \int_{\p^*A(t) \cap \Omega}
(I_d {-} \nu_{A(t)} \otimes \nu_{A(t)}) : \nabla B \,d\H^{d-1} dt
\\&~~~ \nonumber
-  c_0 \int_{0}^{T'} \int_{\p^*A(t) \cap \p\Omega} \cos\alpha
(I_d {-} \nu_{\p\Omega} \otimes \nu_{\p\Omega}) : \nabla B \,d\H^{d-1} dt
\\&~~~ \nonumber
-  c_0 \int_{0}^{T'} \int_{\p^*A(t) \cap \Omega}
(\nu_{A(t)} {-} \xi) \cdot (\p_t\xi {+} (B\cdot\nabla)\xi {+} (\nabla B)^\mathsf{T}\xi) \,d\H^{d-1} dt
\\&~~~ \nonumber
-  c_0 \int_{0}^{T'} \int_{\p^*A(t) \cap \Omega}
\xi \cdot (\p_t\xi {+} (B\cdot\nabla)\xi) \,d\H^{d-1} dt
\\&~~~ \nonumber
-  c_0 \int_{0}^{T'} \int_{\p^*A(t) \cap \Omega}
(\nu_{A(t)}\cdot\xi - 1)(\nabla \cdot B) \,d\H^{d-1} dt
\\&~~~ \nonumber
-  c_0 \int_{0}^{T'} \int_{\p^*A(t) \cap \Omega}
(\nu_{A(t)} {-} \xi) \otimes (\nu_{A(t)} {-} \xi) : \nabla B \,d\H^{d-1} dt.
\end{align}
For convenience of the reader, we will reproduce the argument 
for~\eqref{eq:timeEvolRelEntropyAux0} below. Before that,
however, let us first quickly argue how to deduce~\eqref{eq:timeEvolRelEntropy}
from~\eqref{eq:timeEvolRelEntropyAux0}.
First, plugging in the energy dissipation inequality~\eqref{eq:energyDissipation}
and completing squares three times yields
\begin{align*}
&E(A(T')) - E(A(0))
+  c_0 \int_{0}^{T'} \int_{\p^*A(t) \cap \Omega}
(\nabla\cdot\xi) (V + B\cdot\nu_{A(t)}) \,d\H^{d-1} dt
\\&
\leq -  c_0 \int_{0}^{T'} \int_{\p^*A(t) \cap \Omega}
\frac{1}{2} |V {-} (\nabla\cdot\xi)|^2 \,d\H^{d-1} dt
+  c_0 \int_{0}^{T'} \int_{\p^*A(t) \cap \Omega}
(\nabla\cdot\xi) (B\cdot\nu_{A(t)}) \,d\H^{d-1} dt
\\&~~~
-  c_0 \int_{0}^{T'} \int_{\p^*A(t) \cap \Omega}
\frac{1}{2} |V|^2 \,d\H^{d-1} dt
+  c_0 \int_{0}^{T'} \int_{\p^*A(t) \cap \Omega}
\frac{1}{2} |\nabla\cdot\xi|^2 \,d\H^{d-1} dt
\\&
\leq -  c_0 \int_{0}^{T'} \int_{\p^*A(t) \cap \Omega}
\frac{1}{2} |V {-} (\nabla\cdot\xi)|^2 \,d\H^{d-1} dt
-  c_0 \int_{0}^{T'} \int_{\p^*A(t) \cap \Omega}
\frac{1}{2} |V{+}(B\cdot\xi)|^2 \,d\H^{d-1} dt
\\&~~~
+  c_0 \int_{0}^{T'} \int_{\p^*A(t) \cap \Omega}
\frac{1}{2} |(B\cdot\xi) {+}(\nabla\cdot\xi)|^2 \,d\H^{d-1} dt
\\&~~~
+  c_0 \int_{0}^{T'} \int_{\p^*A(t) \cap \Omega} 
B\cdot(\nu_{A(t)}{-}\xi)(\nabla\cdot\xi) \,d\H^{d-1} dt
+  c_0 \int_{0}^{T'} \int_{\p^*A(t) \cap \Omega}
(B\cdot\xi)V \,d\H^{d-1} dt.
\end{align*}
Second, testing~\eqref{eq:motionLaw} with the velocity~$B$,
which is indeed admissible thanks to~\eqref{eq:angleCondB}, 
and recalling Young's law in form of~$ c_0\cos\alpha=\jump{\hat\sigma}$ entails
\begin{align*}
&-  c_0 \int_{0}^{T'} \int_{\p^*A(t) \cap \Omega}
(I_d {-} \nu_{A(t)} \otimes \nu_{A(t)}) : \nabla B \,d\H^{d-1} dt
\\&~~~
-  c_0 \int_{0}^{T'} \int_{\p^*A(t) \cap \p\Omega} \cos\alpha
(I_d {-} \nu_{\p\Omega} \otimes \nu_{\p\Omega}) : \nabla B \,d\H^{d-1} dt
\\&
= -  c_0 \int_{0}^{T'} \int_{\p^*A(t) \cap \p\Omega} 
B\cdot\nu_{A(t)}V \,d\H^{d-1} dt.
\end{align*}
Inserting back into~\eqref{eq:timeEvolRelEntropyAux0} the right hand sides 
of the previous two displays and closely inspecting
the asserted estimate~\eqref{eq:timeEvolRelEntropy} thus concludes the proof.
\end{proof}

\begin{proof}[Proof of~\eqref{eq:timeEvolRelEntropyAux0}]
As a side remark, we emphasize that the following argument only
relies on using $ c_0(\nabla\cdot\xi)$ as a test function
in the transport equation~\eqref{eq:evolSet} of~$A(t)$,
some generic algebraic manipulations, and several integration
by parts, the latter in particular only exploiting the boundary 
conditions~\eqref{eq:angleCondXi}--\eqref{eq:angleCondB}
and the regularity of~$(\xi,B)$.

For the derivation of~\eqref{eq:timeEvolRelEntropyAux0},
the first important observation is that one may
rewrite the relative entropy~\eqref{eq:relEntropy}
in terms of the energy~\eqref{eq:energySharpInterface} as follows:
\begin{align*}
\E_{\mathrm{relEn}}(A(t)|\mathscr{A}(t))
= E(A(t)) + c_0\int_{A(t)} (\nabla\cdot\xi)(\cdot,t) \,dx
\end{align*}
for all $t \in [0,T]$. Indeed, integrating by parts
in the second term of the relative entropy~\eqref{eq:relEntropy},
making use in the process of the boundary condition~\eqref{eq:angleCondXi},
and finally recalling Young's law in form of $ c_0\cos\alpha=\jump{\hat\sigma}$,
one generates the boundary energy contribution in~\eqref{eq:energySharpInterface} 
as well as the second right hand side term of the previous display.
One then capitalizes on the previous display by testing~\eqref{eq:evolSet}
with $ c_0(\nabla\cdot\xi)$ and integrating by parts to obtain
\begin{align}
\nonumber
&\E_{\mathrm{relEn}}(A(T')|\mathscr{A}(T')) - \E_{\mathrm{relEn}}(A(0)|\mathscr{A}(0))
\\& \label{eq:timeEvolRelEntropyAux1}
= E(A(T')) - E(A(0))
+  c_0 \int_{0}^{T'} \int_{\p^*A(t) \cap \Omega}
(\nabla\cdot\xi)V \,d\H^{d-1} dt
\\&~~~ \nonumber 
-  c_0 \int_{0}^{T'} \int_{\p^*A(t) \cap \Omega}
\nu_{A(t)}\cdot \p_t\xi \,d\H^{d-1} dt
-  c_0 \int_{0}^{T'} \int_{\p^*A(t) \cap \p\Omega}
\nu_{\p\Omega}\cdot \p_t\xi \,d\H^{d-1} dt.
\end{align}
Adding zero several times moreover implies
\begin{align*}
&-  c_0 \int_{0}^{T'} \int_{\p^*A(t) \cap \Omega}
\nu_{A(t)}\cdot \p_t\xi \,d\H^{d-1} dt
-  c_0 \int_{0}^{T'} \int_{\p^*A(t) \cap \p\Omega}
\nu_{\p\Omega}\cdot \p_t\xi \,d\H^{d-1} dt
\\&
= -  c_0 \int_{0}^{T'} \int_{\p^*A(t) \cap \Omega}
\nu_{A(t)}\cdot (\p_t\xi {+} (B\cdot\nabla)\xi {+} (\nabla B)^\mathsf{T}\xi) \,d\H^{d-1} dt
\\&~~~
-  c_0 \int_{0}^{T'} \int_{\p^*A(t) \cap \p\Omega}
\nu_{\p\Omega}\cdot (\p_t\xi + (B\cdot\nabla)\xi) \,d\H^{d-1} dt
\\&~~~
+  c_0 \int_{0}^{T'} \int_{\p^*A(t) \cap \Omega}
\xi \otimes \nu_{A(t)} : \nabla B\,d\H^{d-1} dt
\\&~~~
+  c_0 \int_{0}^{T'} \int_{\p^*A(t) \cap \Omega}
\nu_{A(t)}\cdot (B\cdot\nabla)\xi \,d\H^{d-1} dt
+  c_0 \int_{0}^{T'} \int_{\p^*A(t) \cap \p\Omega}
\nu_{\p\Omega}\cdot (B\cdot\nabla)\xi \,d\H^{d-1} dt
\\&
= -  c_0 \int_{0}^{T'} \int_{\p^*A(t) \cap \Omega} (\nu_{A(t)} {-} \xi) \cdot 
(\p_t\xi {+} (B\cdot\nabla)\xi {+} (\nabla B)^\mathsf{T}\xi) \,d\H^{d-1} dt
\\&~~~
-  c_0 \int_{0}^{T'} \int_{\p^*A(t) \cap \Omega} \xi \cdot 
(\p_t\xi {+} (B\cdot\nabla)\xi) \,d\H^{d-1} dt
\\&~~~
-  c_0 \int_{0}^{T'} \int_{\p^*A(t) \cap \p\Omega}
\nu_{\p\Omega}\cdot (\p_t\xi + (B\cdot\nabla)\xi) \,d\H^{d-1} dt
\\&~~~
+  c_0 \int_{0}^{T'} \int_{\p^*A(t) \cap \Omega}
\xi \otimes (\nu_{A(t)} {-} \xi) : \nabla B\,d\H^{d-1} dt
\\&~~~
+  c_0 \int_{0}^{T'} \int_{\p^*A(t) \cap \Omega}
\nu_{A(t)}\cdot (B\cdot\nabla)\xi \,d\H^{d-1} dt
+  c_0 \int_{0}^{T'} \int_{\p^*A(t) \cap \p\Omega}
\nu_{\p\Omega}\cdot (B\cdot\nabla)\xi \,d\H^{d-1} dt.
\end{align*}
By the boundary condition~\eqref{eq:angleCondXi},
we obviously have $\nu_{\p\Omega}\cdot \p_t\xi = 0$
along $\p\Omega{\times}[0,T]$. Applying the product rule
and the boundary conditions~\eqref{eq:angleCondXi}--\eqref{eq:angleCondB} in form of 
$\nu_{\p\Omega}\cdot (B\cdot\nabla)\xi = - \xi\cdot (B\cdot\nabla)\nu_{\p\Omega}$,
and recalling the well-known fact that $(\nabla^\mathrm{tan}\nu_{\p\Omega})^\mathsf{T}\nu_{\p\Omega}=0$,
we get
\begin{align}
\nonumber
&-  c_0 \int_{0}^{T'} \int_{\p^*A(t) \cap \p\Omega}
\nu_{\p\Omega}\cdot (\p_t\xi + (B\cdot\nabla)\xi) \,d\H^{d-1} dt
\\& \label{eq:timeEvolRelEntropyAux2}
=  c_0 \int_{0}^{T'} \int_{\p^*A(t) \cap \p\Omega}
((I_d {-} \nu_{\p\Omega} \otimes \nu_{\p\Omega})\xi)
\cdot (B\cdot\nabla)\nu_{\p\Omega} \,d\H^{d-1} dt.
\end{align}
Next, making use of the product rule in form of $(B\cdot\nabla)\xi
= \nabla\cdot(\xi\otimes B) - \xi(\nabla\cdot B)$ and adding zero
two times entails
\begin{align*}
& c_0 \int_{0}^{T'} \int_{\p^*A(t) \cap \Omega}
\nu_{A(t)}\cdot (B\cdot\nabla)\xi \,d\H^{d-1} dt
+  c_0 \int_{0}^{T'} \int_{\p^*A(t) \cap \p\Omega}
\nu_{\p\Omega}\cdot (B\cdot\nabla)\xi \,d\H^{d-1} dt
\\&
=  c_0 \int_{0}^{T'} \int_{\p^*A(t) \cap \Omega}
\nu_{A(t)} \cdot \nabla\cdot(\xi\otimes B) \,d\H^{d-1} dt
+  c_0 \int_{0}^{T'} \int_{\p^*A(t) \cap \p\Omega}
\nu_{\p\Omega} \cdot \nabla\cdot(\xi\otimes B) \,d\H^{d-1} dt
\\&~~~
-  c_0 \int_{0}^{T'} \int_{\p^*A(t) \cap \Omega}
(I_d {-} \nu_{A(t)} \otimes \nu_{A(t)}) : \nabla B \,d\H^{d-1} dt
-  c_0 \int_{0}^{T'} \int_{\p^*A(t) \cap \p\Omega}
(\nu_{\p\Omega} \cdot \xi)(\nabla\cdot B) \,d\H^{d-1} dt
\\&~~~
-  c_0 \int_{0}^{T'} \int_{\p^*A(t) \cap \Omega}
(\nu_{A(t)} \cdot \xi - 1) (\nabla \cdot B) \,d\H^{d-1} dt
-  c_0 \int_{0}^{T'} \int_{\p^*A(t) \cap \Omega}
\nu_{A(t)} \otimes \nu_{A(t)} : \nabla B \,d\H^{d-1} dt.
\end{align*}
We further compute based on an integration by parts,
the symmetry relation $\nabla\cdot(\nabla\cdot(\xi \otimes B))
= \nabla\cdot(\nabla\cdot(B \otimes \xi))$, reverting
the integration by parts, $\nabla\cdot(B \otimes \xi)
= (\xi\cdot\nabla)B + B(\nabla\cdot\xi)$, and the 
boundary condition~\eqref{eq:angleCondB}
\begin{align*}
& c_0 \int_{0}^{T'} \int_{\p^*A(t) \cap \Omega}
\nu_{A(t)} \cdot \nabla\cdot(\xi\otimes B) \,d\H^{d-1} dt
+  c_0 \int_{0}^{T'} \int_{\p^*A(t) \cap \p\Omega}
\nu_{\p\Omega} \cdot \nabla\cdot(\xi\otimes B) \,d\H^{d-1} dt
\\&
=  c_0 \int_{0}^{T'} \int_{A(t)}
\nabla\cdot(\nabla\cdot(B \otimes \xi)) \,dx dt
\\&
=  c_0 \int_{0}^{T'} \int_{\p^*A(t) \cap \Omega}
(\nabla\cdot\xi)(B \cdot \nu_{A(t)}) \,d\H^{d-1} dt
+  c_0 \int_{0}^{T'} \int_{\p^*A(t) \cap \Omega}
\nu_{A(t)} \otimes \xi : \nabla B \,d\H^{d-1} dt
\\&~~~
+  c_0 \int_{0}^{T'} \int_{\p^*A(t) \cap \p\Omega}
\nu_{\p\Omega} \otimes \xi : \nabla B \,d\H^{d-1} dt.
\end{align*}
Splitting $\xi= (\nu_{\p\Omega}\cdot\xi)\nu_{\p\Omega}
+ (I_d{-}\nu_{\p\Omega}\otimes\nu_{\p\Omega})\xi$,
exploiting the product rule, the boundary condition~\eqref{eq:angleCondB}
and the symmetry of $\nabla^{\mathrm{tan}}\nu_{\p\Omega}$
in form of the identity $\nu_{\p\Omega} \otimes ((I_d{-}\nu_{\p\Omega}\otimes\nu_{\p\Omega})\xi) : \nabla B
= - ((I_d {-} \nu_{\p\Omega} \otimes \nu_{\p\Omega})\xi)
\cdot (B\cdot\nabla)\nu_{\p\Omega}$, we may rewrite
\begin{align}
\nonumber
& c_0 \int_{0}^{T'} \int_{\p^*A(t) \cap \p\Omega}
\nu_{\p\Omega} \otimes \xi : \nabla B \,d\H^{d-1} dt
\\& \label{eq:timeEvolRelEntropyAux3}
=  c_0 \int_{0}^{T'} \int_{\p^*A(t) \cap \p\Omega}
(\nu_{\p\Omega}\cdot\xi) \, \nu_{\p\Omega} \otimes \nu_{\p\Omega} : \nabla B \,d\H^{d-1} dt
\\&~~~ \nonumber
-  c_0 \int_{0}^{T'} \int_{\p^*A(t) \cap \p\Omega}
((I_d {-} \nu_{\p\Omega} \otimes \nu_{\p\Omega})\xi)
\cdot (B\cdot\nabla)\nu_{\p\Omega} \,d\H^{d-1} dt.
\end{align}
In particular, the combination of~\eqref{eq:timeEvolRelEntropyAux2}, \eqref{eq:timeEvolRelEntropyAux3},
and~\eqref{eq:angleCondXi} yields
\begin{align*}
&-  c_0 \int_{0}^{T'} \int_{\p^*A(t) \cap \p\Omega}
\nu_{\p\Omega}\cdot (\p_t\xi + (B\cdot\nabla)\xi) \,d\H^{d-1} dt
+  c_0 \int_{0}^{T'} \int_{\p^*A(t) \cap \p\Omega}
\nu_{\p\Omega} \otimes \xi : \nabla B \,d\H^{d-1} dt
\\&
-  c_0 \int_{0}^{T'} \int_{\p^*A(t) \cap \p\Omega}
(\nu_{\p\Omega} \cdot \xi)(\nabla\cdot B) \,d\H^{d-1} dt
\\&
= -  c_0 \int_{0}^{T'} \int_{\p^*A(t) \cap \p\Omega} \cos\alpha
(I_d {-} \nu_{\p\Omega} \otimes \nu_{\p\Omega}) : \nabla B \,d\H^{d-1} dt.
\end{align*}
Inserting back into~\eqref{eq:timeEvolRelEntropyAux1} the information 
provided by the previous six displays finally allows to conclude.
\end{proof}

\subsection{Quantitative stability of the bulk error}
In the setting of Theorem~\ref{theo:weakStrongUniqueness},
we claim that for a.e.\ $T'\in [0,T]$ it holds
\begin{align}
\label{eq:timeEvolBulkError}
\E_{\mathrm{bulk}}(A(T')|\mathscr{A}(T'))
&= \E_{\mathrm{bulk}}(A(0)|\mathscr{A}(0))
+ \int_{0}^{T'} \int_{\Omega} (\chi_{A(t)} {-} \chi_{\mathscr{A}(t)}) 
\vartheta(\nabla\cdot B) \,dx dt
\\&~~~ \nonumber
+ \int_{0}^{T'} \int_{\Omega} (\chi_{A(t)} {-} \chi_{\mathscr{A}(t)}) 
(\partial_t\vartheta {+} (B\cdot\nabla)\vartheta) \,dx dt
\\&~~~ \nonumber
+ \int_{0}^{T'} \int_{\p^*A(t) \cap \Omega} 
\vartheta B\cdot(\nu_{A(t)} {-} \xi) \,d\H^{d-1} dt
\\&~~~ \nonumber
+ \int_{0}^{T'} \int_{\p^*A(t) \cap \Omega} 
\vartheta(V + B\cdot\xi) \,d\H^{d-1} dt.
\end{align}

\begin{proof}[Proof of~\eqref{eq:timeEvolBulkError}]
In order to compute the time evolution of the 
bulk error~\eqref{eq:bulkError}, we first note that
thanks to the sign conditions~\eqref{eq:signWeightExterior}--\eqref{eq:signWeightInterior} 
we simply have
\begin{align*}
\E_{\mathrm{bulk}}(A(t)|\mathscr{A}(t))
= \int_{\Omega} (\chi_{A(t)} {-} \chi_{\mathscr{A}(t)}) 
\vartheta(\cdot,t) \,dx.
\end{align*}
Hence, plugging in $\vartheta$ as a test function in~\eqref{eq:evolSet}
and using the conditions~\eqref{eq:signWeightInterface}
and $\partial_t\chi \ll |\nabla\chi| \llcorner \big(\bigcup_{t\in [0,T]}
(\p^*\mathscr{A}(t) {\cap} \Omega) {\times} \{t\} \big)$, it follows
\begin{align*}
\E_{\mathrm{bulk}}(A(T')|\mathscr{A}(T'))
&= \E_{\mathrm{bulk}}(A(0)|\mathscr{A}(0))
+ \int_{0}^{T'} \int_{\Omega} (\chi_{A(t)} {-} \chi_{\mathscr{A}(t)}) 
\partial_t\vartheta \,dx dt 
\\&~~~
+ \int_{0}^{T'} \int_{\p^*A(t) \cap \Omega} V\vartheta \,d\H^{d-1} dt.
\end{align*}
Adding zero to generate the advective derivative of the weight~$\vartheta$,
appealing to the product rule in form of $(B\cdot\nabla)\vartheta = \nabla\cdot(B\vartheta)
- \vartheta(\nabla\cdot B)$, and integrating by parts using in particular
the boundary condition~\eqref{eq:angleCondB} for~$B$ and again
the condition~\eqref{eq:signWeightInterface}, we obtain
\begin{align*}
\int_{0}^{T'} \int_{\Omega} (\chi_{A(t)} {-} \chi_{\mathscr{A}(t)}) 
\partial_t\vartheta \,dx dt 
&= \int_{0}^{T'} \int_{\Omega} (\chi_{A(t)} {-} \chi_{\mathscr{A}(t)}) 
(\partial_t\vartheta {+} (B\cdot\nabla)\vartheta) \,dx dt
\\&~~~ 
+ \int_{0}^{T'} \int_{\Omega} (\chi_{A(t)} {-} \chi_{\mathscr{A}(t)}) 
\vartheta(\nabla\cdot B) \,dx dt 
\\&~~~
+ \int_{0}^{T'} \int_{\p^*A(t) \cap \Omega} 
\vartheta(B\cdot\nu_{A(t)}) \,d\H^{d-1} dt.
\end{align*}
The previous two displays obviously imply the claim.
\end{proof}

\subsection{Proof of Theorem~\ref{theo:weakStrongUniqueness}}	
Since $|\xi|\leq 1$ due to~\eqref{eq:coercivityLengthXi}
we have $|\nu_{A(t)}{-}\xi(\cdot,t)| \leq
2(1-\nu_{A(t)}\cdot\xi(\cdot,t))$ for all $t\in [0,T]$.
Due to the regularity~\eqref{eq_regularityB} and the estimate~\eqref{eq:timeEvolWeight},
it thus follows from~\eqref{eq:timeEvolBulkError} that
\begin{align*}
\E_{\mathrm{bulk}}(A(T')|\mathscr{A}(T'))
&\leq \E_{\mathrm{bulk}}(A(0)|\mathscr{A}(0))
+ C\int_{0}^{T'} (\E_{\mathrm{bulk}} {+} \E_{\mathrm{relEn}})(A(t)|\mathscr{A}(t)) \,dt
\\&~~~ \nonumber
+ \int_{0}^{T'} \int_{\p^*A(t) \cap \Omega} 
|\vartheta||V {+} (B\cdot\xi)| \,d\H^{d-1} dt.
\end{align*}
Furthermore, note that~$|\vartheta|(\cdot,t) \leq 
C\min\{1,\dist(\cdot,\overline{\p^*\mathscr{A}(t)\cap\Omega})\}$ for all $t\in [0,T]$
due to~\eqref{eq:signWeightInterface} and~\eqref{eq_regularityWeight}.
Hence, as a consequence of Young's inequality, the coercivity condition~\eqref{eq:coercivityLengthXi},
and the definition~\eqref{eq:relEntropy} it holds for all $\delta\in (0,1]$
\begin{align}
\label{eq:timeEvolBulkErrorPostProcessed1}
\E_{\mathrm{bulk}}(A(T')|\mathscr{A}(T'))
&\leq \E_{\mathrm{bulk}}(A(0)|\mathscr{A}(0)) 
+ C(\delta)\int_{0}^{T'} (\E_{\mathrm{bulk}} {+} \E_{\mathrm{relEn}})(A(t)|\mathscr{A}(t)) \,dt
\\&~~~ \nonumber
+ \delta \int_{0}^{T'} \int_{\p^*A(t) \cap \Omega} 
|V {+} (B\cdot\xi)|^2 \,d\H^{d-1} dt.
\end{align}
Moreover, by an application of Young's inequality, the regularity~\eqref{eq_regularityB},
the approximate evolution equations~\eqref{eq:timeEvolXi}--\eqref{eq:timeEvolLengthXi},
the condition~\eqref{eq:motionByMCFCalibration}, and again the estimate 
$|\nu_{A(t)}{-}\xi(\cdot,t)| \leq 2(1-\nu_{A(t)}\cdot\xi(\cdot,t))$
we obtain for all $\delta\in (0,1]$
\begin{align}
\nonumber
&\E_{\mathrm{relEn}}(A(T')|\mathscr{A}(T')) + 
 c_0 \int_{0}^{T'} \int_{\p^*A(t) \cap \Omega}
\frac{1}{2} |V {-} (\nabla\cdot\xi)|^2 \,d\H^{d-1} dt
\\&\nonumber
+  c_0 \int_{0}^{T'} \int_{\p^*A(t) \cap \Omega}
\frac{1}{2} |V{+}(B\cdot\xi)|^2 \,d\H^{d-1} dt
\\& \label{eq:timeEvolRelEntropyPostProcessed1}
\leq \E_{\mathrm{relEn}}(A(0)|\mathscr{A}(0))
+ C(\delta)\int_{0}^{T'} \E_{\mathrm{relEn}}(A(t)|\mathscr{A}(t))\,dt 
\\&~~~ \nonumber
+ \delta \int_{0}^{T'} \int_{\p^*A(t) \cap \Omega} 
|V {-} (\nabla\cdot\xi)|^2 \,d\H^{d-1} dt.
\end{align}
Hence, by an absorption we get~\eqref{eq:stabilityEstimateRelEntropy}
in the stronger form of
\begin{align}
\nonumber
&\E_{\mathrm{relEn}}(A(T')|\mathscr{A}(T')) + 
c \int_{0}^{T'} \int_{\p^*A(t) \cap \Omega}
\frac{1}{2} |V {-} (\nabla\cdot\xi)|^2 \,d\H^{d-1} dt
\\&\nonumber
+  c_0 \int_{0}^{T'} \int_{\p^*A(t) \cap \Omega}
\frac{1}{2} |V{+}(B\cdot\xi)|^2 \,d\H^{d-1} dt
\\& \label{eq:timeEvolRelEntropyPostProcessed2}
\leq \E_{\mathrm{relEn}}(A(0)|\mathscr{A}(0)) 
+ C\int_{0}^{T'} \E_{\mathrm{relEn}}(A(t)|\mathscr{A}(t)) \,dt
\end{align}
for some constants $c\in (0,1)$ and $C>0$.
Adding~\eqref{eq:timeEvolRelEntropyPostProcessed2} to~\eqref{eq:timeEvolBulkErrorPostProcessed1}
in combination with another absorption finally entails~\eqref{eq:stabilityEstimateBulkError}. \qed

\section*{Acknowledgments}
This project has received funding from the European Research Council 
(ERC) under the European Union's Horizon 2020 research and innovation 
programme (grant agreement No 948819)
\begin{tabular}{@{}c@{}}\includegraphics[width=8ex]{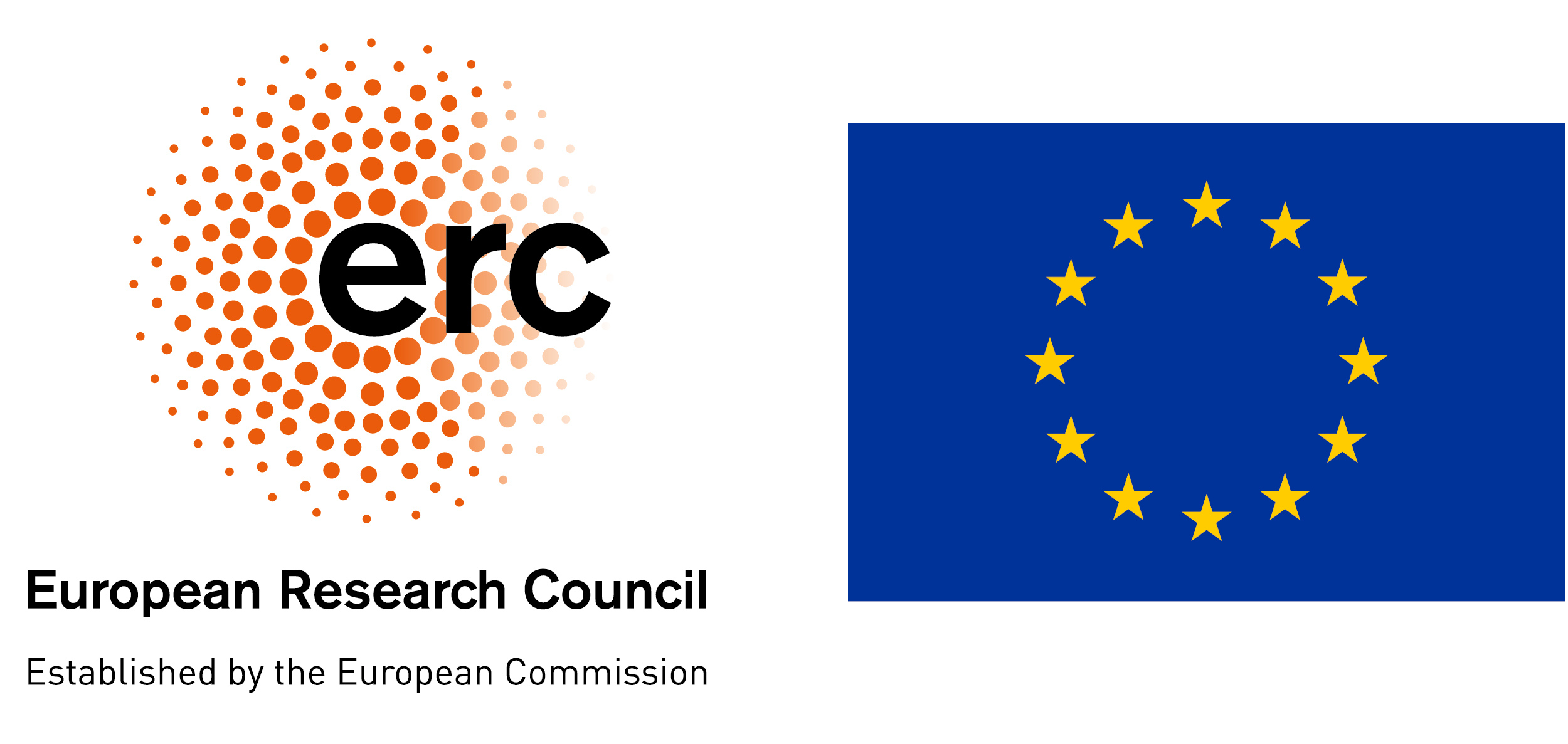}\end{tabular}, and 
from the Deutsche Forschungsgemeinschaft (DFG, German Research Foundation) under 
Germany's Excellence Strategy -- EXC-2047/1 -- 390685813.
Parts of this paper were developed during a visit of the first author to the Hausdorff Center of Mathematics (HCM),
University of Bonn. The hospitality and the support of HCM are gratefully acknowledged. 

\bibliographystyle{abbrv}
\bibliography{ac-to-bv-contact-angle}
	
  \end{document}